\documentclass [a4paper]{article}
\usepackage{amsmath}
\usepackage{mathrsfs}
\usepackage{amsfonts}
\usepackage {graphicx}
\usepackage{color,xcolor}
\usepackage{amsmath,amssymb}
\input{epsf.sty}
\usepackage{epsfig}

\usepackage{amsthm}
\newtheorem{thm}{{\CC 6( {\char64}m}}[section]

\newtheorem{prop}[thm]{Proposition}

\newtheorem{rem}{Remark}[section]
\newtheorem{theorem}{Theorem}[section]

\newtheorem{algorithm}{Algorithm}
\def\be{\begin{equation}}
\def\ee{\end{equation}}
\def\br{\begin{eqnarray*}}
\def\er{\end{eqnarray*}}

\begin{document}
%\renewcommend{theequation}{\thesection.\arabic{equation}}
%\newcommend{\be}{begin{equation}}
%\newcommend{\\ee}{\end{equation}}
\begin{center}
\Large\bf{A framework of the harmonic Arnoldi method for evaluating $\varphi$-functions with applications to exponential integrators
}\\
%
%\renewcommend{\thefootnote}{\fnsymbol{footnote}}
\quad\\

\normalsize~
Gang Wu\footnote[1]{Corresponding author (G. Wu). Department of Mathematics,
China University of Mining and Technology \& School of Mathematics and Statistics, Jiangsu Normal University, Xuzhou, 221116, Jiangsu, P.R. China.
E-mail: {\tt gangwu76@126.com} and {\tt wugangzy@gmail.com}. This author is
supported by the National Science Foundation of China under grant 11371176, the National Science Foundation of Jiangsu Province under grant BK20131126, the 333 Project of Jiangsu Province, and the Talent Introduction Program of China University of Mining and Technology.},
~Lu Zhang\footnote[2]{Department of Mathematics, University of Macau, Macao, P.R. China.
Email: {\tt yulu7517@126.com}.},
~Ting-ting Xu\footnote[3]{School of Mathematics and Statistics, Jiangsu Normal University, Xuzhou, 221116, Jiangsu, P.R. China.
Email: {\tt xutingtingdream@163.com}. This author is supported by the Postgraduate Innovation Project of Jiangsu Normal University under grant 2013YYB110.}
\end{center}

\begin{abstract}
In recent years, a great deal of attention has been focused on numerically solving exponential
integrators.
The important ingredient to the implementation of exponential integrators is the efficient and accurate evaluation of the so called $\varphi$-functions on a given vector. The Krylov subspace method is an important technique for this problem. For this type of method, however, restarts become essential for the sake of storage requirements
or due to the growing computational complexity of evaluating the matrix function on a Hessenberg matrix of growing
size. Another problem in computing $\varphi$-functions is the lack of
a clear residual notion. The contribution of this work is threefold. First, we introduce a framework of the harmonic Arnoldi method for $\varphi$-functions, which is based on the residual and the oblique projection technique. Second, we establish the relationship between the harmonic Arnoldi approximation and the classical Arnoldi approximation, and compare the harmonic Arnoldi method with the Arnoldi method from a theoretical point of view. Third, we apply the thick-restarting strategy to the harmonic Arnoldi method, and propose a thick-restated harmonic Arnoldi algorithm for evaluating $\varphi$-functions. An advantage of the new algorithm is that we can compute several $\varphi$-functions simultaneously in the same search subspace. We show the merit of augmenting approximate eigenvectors in the search subspace, and give insight into the relationship between the error and the residual of $\varphi$-functions. Numerical experiments show the superiority of our new algorithm over many state-of-the-art algorithms for the computation of $\varphi$-functions.

\mbox{\bf Keywords:}  Exponential integrators, $\varphi$-functions, Matrix exponential, Harmonic Arnoldi method, Oblique projection method, Thick-restarting strategy.

\mbox{\bf AMS classifications: } 65F60, 65F15, 65F10.
\end{abstract}

\section{Introduction}

\setcounter{equation}{0}

Exponential integrators have been employed in various large scale
computations \cite{Higham,Hochbruck-Ostermann-2010}, such as reaction-diffusion systems \cite{Fri}, mathematical finance \cite{Rag}, classical and quantum-classical molecular dynamics \cite{Sch},
Schr\"{o}dinger equations \cite{Ber}, Maxwell equations \cite{Bot}, regularization of ill-posed problems \cite{Hoc1}, and so on.
The key to the implementation of exponential integrators is the efficient and accurate
evaluation of the matrix exponential and other $\varphi$-functions. These $\varphi$-functions
are defined for scalar arguments by the integral representation
\begin{equation}\label{equ1}
%$$
\varphi_{0}(z)={e}^{z}~~{\rm and}\quad\varphi_{\ell}(z)=\frac{1}{(\ell-1)!}{\int_0^1 e^{(1-\theta)z}\theta^{\ell-1}d\theta},\quad \ell=1,2,\ldots,\quad z\in \mathbb{C},
%$$
\end{equation}
moreover, these $\varphi$-functions satisfy the following recurrence relation
\begin{equation}\label{eqn11}
\varphi_{\ell}(z)=z\varphi_{\ell+1}(z)+\frac{1}{\ell!},\quad \ell=0,1,2,\ldots
\end{equation}
The definition can be extended to matrices instead of scalars using any of the
available definitions of matrix functions \cite{Hochbruck-Ostermann-2010,NW2}.

Exponential integrators constitute an interesting class of numerical methods
for the time integration of stiff systems of differential equations. The so-called $\varphi$-functions and their evaluation are crucial for stability and speed of exponential integrators.
The important ingredient to implementation of exponential integrators is the computation of the matrix exponential
and related $\varphi$-functions on a given vector \cite{Hochbruck-Ostermann-2010,NW2}.
In some practical applications, it is required to compute a few $\varphi$-functions on a given vector \cite{Hochbruck-Ostermann-2010}
\begin{equation}\label{eqn1.3}
{\bf y}(t)=\varphi_{\ell}(-tA){\bf v},\quad \ell=0,1,\ldots,s,
\end{equation}
where $A$ is a large scale matrix and $s\geq 0$ is a user-prescribed parameter. In this paper, we assume that $-tA$ is semi-negative definite, i.e., the real part of the spectrum of $-tA$ lies in the left half plane, and we are interested in solving the $(s+1)$ vectors {\it simultaneously} in the {\it same} search subspace.

Recently, a great deal of attention has been focused on numerical solution of exponential integrators.
For instance, a MATLAB package called EXPINT \cite{Ber2} is provided which aims to facilitate the quick deployment and testing of exponential
integrators.
This approach is based on a modification of
the scaling and squaring technique for
the matrix exponential \cite{Higham,Higham-2009,B-W-2009}, and is suitable for $\varphi$-functions of medium sized matrices. In \cite{ST}, Schmelzer and Trefethen show that the $\varphi$-functions can be evaluated by using rational approximations constructed via Carath\'{e}odory-Fej\'{e}r
approximation or contour integrals.
In \cite{NW2}, an adaptive Krylov subspace
algorithm is proposed for evaluating the
$\varphi$-functions appearing in exponential integrators.
The {\it phipm} function is given for calculating the action of linear
combinations of $\varphi_{\ell}$ on operand vectors, and it can be
considered as an extension of the codes provided in EXPOKIT \cite{Sidje}.
We refer to the review paper \cite{Hochbruck-Ostermann-2010} and the references therein for
the properties of exponential integrators and some efficient numerical methods for solving them.

The Krylov subspace methods are popular techniques for the computation of $\varphi$-functions \cite{E-Y-1992,T-V-2013,Goc,Higham,M-C-1997,Hoc,Hochbruck-Ostermann-2010,NW1,NW2,Rag,WW}, in which
the Arnoldi method is a widely used one \cite{NW1,NW2,Saad,Sidje}. In this method, the large matrix $A$ is projected into
a much smaller subspace, then the matrix function is applied to the reduced matrix (or the projection matrix),
and finally the approximation is projected back to the original large space.
However, the maximum number of iterations that can be performed is often limited by
the storage requirements of the full Arnoldi basis.
A further limiting factor is the growing orthogonalization cost of computing the Arnoldi basis and the cost of evaluating
the matrix function of the projection matrix for larger values of Arnoldi steps.

In order to overcome these difficulties, several alternative approaches have also been proposed. The first one is
to use other subspaces with superior approximation
properties,
such as the extended Krylov subspace methods \cite{V-L-1998,T-V-2013,L-V-2010} or the shift-and-invert Krylov subspace methods \cite{Goc,IM,Pang,MP,Rag,vH}.
Both of them can be viewed as special cases of the rational Krylov subspace methods \cite{B-L-2009,Goc,S-2013,S-L-2013,Novati-2011}, with the aim to reach a targeted accuracy within significantly fewer iterations.
However, the rational Krylov subspace methods require to solve a (shifted) linear system
at each Arnoldi step, which is a major drawback for situations
when $A$ is large or the matrix is not explicitly available but only implicitly as a
routine returning matrix-vector products.

The other possible approach for circumventing the problems mentioned above is based on restarting. The restarted Krylov
subspace methods \cite{ETNA,MA,SINUM,EEG,IMA,SISC} restart the Arnoldi process periodically,
to avoid storing large sets of Arnoldi basis vectors. In \cite{EEG}, a deflated restarting technique was proposed to accelerate the convergence of the restarted Arnoldi approximation. Its effect is to ultimately deflate a specific invariant subspace of the matrix which most impedes the convergence of the restarted Arnoldi approximation process. Recently,
Frommer {\it et al.} utilized an integral representation for the error of the iterates in the Arnoldi method, and
developed a quadrature-based algorithm with deflated restarting \cite{From}.
However, as was pointed out in \cite{From}, none of the restarting approaches for
general matrix functions was completely satisfactory until now. For instance, all of these variants
may solve the storage problem for the Arnoldi basis, but still have to suffer from operating complexity, growing
cost per restart cycle \cite{EEG}, numerical instability \cite{IMA}, and so on.

Another difficulty arises in the computation of matrix functions is the lack of
a clear residual notion.
The residual can provide a reliable stopping criterion, moreover, it can be used to restart the
iterative methods.
For the matrix exponential function in connection with Krylov approximation, the residual
expression can be found in \cite{Residual,EC,M-C-1997}. In \cite{Car}, one can find a discussion of the residual for the $\varphi_{1}$ function with respect to a Krylov approximation.
Recently, Kandolf {\it et al.} \cite{Kan} considered a residual-based error estimate for Leja interpolation of $\varphi$-functions.

In recent years, special attention has been paid to the harmonic Arnoldi method for matrix functions.
In \cite{Harmonic}, Hochbruck and Hochstenbach reviewed three different derivations of the harmonic Ritz approach for matrix functions.
The idea behind the harmonic Ritz approximation is that for some functions, a particular target
may be important \cite{Harmonic}.
More precisely, it is desirable to deflate some eigenvalues close to a given shift, so that the convergence speed can be improved \cite{EEG,Frommer}.
In \cite{Car}, the harmonic Ritz approach was applied to the computation of $\varphi_1$ matrix function. The harmonic Ritz approach was investigated in \cite{Frommer} for the convergence of restarted Krylov subspace method for Stieltjes functions of matrices. To our best knowledge, however, the relation between the harmonic Arnoldi approximation and the Arnoldi approximation is still unknown.

In this paper, we investigate the residual of the $\varphi$-functions, and introduce a harmonic Arnoldi method for (\ref{eqn1.3}) that is based on the residual and the oblique projection technique. Second, we establish the relationship between the harmonic Arnoldi approximation and the classical Arnoldi approximation, and compare the harmonic Arnoldi method with the Arnoldi method from a theoretical point of view. Furthermore, we apply the thick-restarting strategy \cite{WS} to the harmonic Arnoldi method, and propose a thick-restated harmonic Arnoldi algorithm for evaluating the $\varphi$-functions. An advantage of this new algorithm is that one can evaluate the $(s+1)$ vectors in (\ref{eqn1.3}) {\it simultaneously}, and solve them in the {\it same} search subspace. We show the merit of augmenting approximate eigenvectors in the thick-restarting strategy, and give insight into the relation between the error and the residual of the harmonic Arnoldi approximation. Numerical experiments show the efficiency of our new algorithm and its superiority over many state-of-the-art algorithms for $\varphi$-functions.

This paper is organized as follows. In section 2, we briefly introduce
the Arnoldi method and shift-and-invert Arnoldi method for the computation of $\varphi$-functions. In section 3, we
focus on the harmonic Arnoldi method and investigate the relationship between the harmonic Arnoldi approximation
and the classical Arnoldi approximation. Moreover, we propose a thick-restarted harmonic Arnoldi algorithm which can be used to solve the $(s+1)$ vectors
in (\ref{eqn1.3}) simultaneously.
The relationship between the error and the residual of the harmonic Arnoldi approximation is derived in section 4. In Section 5,
we show the advantage of augmenting approximate eigenvectors in the search subspace of a Krylov subspace method. Numerical experiments are reported in Section 6.

\section{The Arnoldi and the shift-and-invert Arnoldi methods for $\varphi$-functions}

\setcounter{equation}{0}

%The Krylov subspace method has become an important tool for computing
%matrix functions \cite{Higham,ML,Saad}.
In this section, we briefly introduce the Arnoldi method and the shift-and-invert Arnoldi method for $\varphi$-functions, and investigate the residuals of the approximations obtained from these two approaches. We show that the Arnoldi method is an orthogonal projection method, while the shift-and-invert Arnoldi method is an oblique projection method.

\subsection{The Arnoldi and the shift-and-invert Arnoldi methods for matrix exponential}

In this subsection, we consider the action of the $\varphi_0$ matrix function (or the matrix exponential) on a given vector
%\begin{equation}\label{equ11}
$$
{\bf y}(t)=\varphi_0(-tA){\bf v}={\rm exp}(-tA){\bf v}.
$$
%\end{equation}
%where $-tA$ is a negative semi-definite matrix, $t$ is a scalar, and ${\bf v}$ is a non-zero vector.
Let ${\bf v}_1={\bf v}/\|{\bf v}\|_2$, then in exact arithmetic, the $k$-step Arnoldi process will generate an orthonormal basis $V_{k+1}=[{\bf v}_1,{\bf v}_2,\ldots,{\bf v}_{k+1}]$ for the Krylov subspace
$\mathcal{K}_{k+1}(A,{\bf v}_1)={\rm span}\{{\bf v}_1,A{\bf v}_1,\ldots,{A^{k}}{\bf v}_1\}$. The following Arnoldi relation holds \cite{Stewart}
\begin{equation}\label{equ22}
 AV_{k}=V_{k}H_{k}+h_{k+1,k}{\bf v}_{k+1}{\bf e}_{k}^{\rm H},
\end{equation}
where $H_{k}$ is a $k$-by-$k$ upper-Hessenberg matrix,
${\bf e}_{k}\in \mathbb{R}^{k}$ is the $k$-th column of the identity matrix, and $(\cdot)^{\rm H}$ denotes the conjugate transpose of a vector or matrix. The Arnoldi method makes use of \cite{Saad}
$$
{\bf y}_{k}(t)=V_{k}{\rm exp}(-tH_{k})\beta{\bf e}_{1}\equiv V_k{\bf u}_k(t),
$$
as an approximation to ${\bf y}(t)$, where ${\bf u}_k(t)={\rm exp}(-tH_{k})\beta{\bf e}_{1}$ and $\beta=\|{\bf v}\|_{2}$.
Notice that
${\bf u}_{k}'(t)=-H_{k}{\bf u}_{k}(t)$ and ${\bf u}_{k}(0)=\beta {\bf e}_{1}$,
thus
$$
{\bf y}_{k}'(t)=V_k{\bf u}_k'(t)=-V_{k}H_{k}{\rm exp}(-t H_{k})\beta{\bf e}_{1}.
$$
It follows from (\ref{equ22}) that the residual is \cite{Residual}
\begin{equation}\label{equ24}
{\bf r}_{k}(t)=-A{\bf y}_{k}(t)-{\bf y}_{k}'(t)=- h_{k+1,k}\Big[{\bf e}_{k}^{\rm H}{\rm exp}(-t H_{k})\beta{\bf e}_{1}\Big]{\bf v}_{k+1},
\end{equation}
and
$$
\|{\bf r}_{k}(t)\|_2=\Big| h_{k+1,k}\big[{\bf e}_{k}^{\rm H}{\rm exp}(-t H_{k})\beta{\bf e}_{1}\big]\Big|.
$$
We see from (\ref{equ24}) that the residual vector ${\bf r}_{k}(t)$ is colinear with the $(k+1)$-th basis vector ${\bf v}_{k+1}$, and it is orthogonal to the search space ${\rm span}\{V_k\}$, i.e.,
\begin{equation}\label{eqn2.5}
\left\{\begin{array}{c}
{\bf y}_{k}(t)=V_{k}{\rm exp}(-tH_{k})\beta{\bf e}_{1}~\in~{\rm span}\{V_{k}\},\\
-A{\bf y}_{k}(t)-{\bf y}_{k}'(t)~\bot~{\rm span}\{V_{k}\}.
\end{array}\right.
\end{equation}
Thus, the Arnoldi method for matrix exponential is an orthogonal projection method \cite{Residual,Stewart}.

In recent works on the approximations of matrix functions by Krylov subspace methods, it becomes more and more apparent that the
shift-and-invert Arnoldi method works tremendously better than the standard Arnoldi method  \cite{Goc,IM,Pang,MP,Rag,vH}.
In this type of method, the Krylov subspace is generated by using the matrix $(I+\gamma A)^{-1}$ instead of $A$, where $\gamma$ is a user-described parameter. Let $\widetilde{\bf v}_1={\bf v}/\|{\bf v}\|_2$, then in exact arithmetic, the $k$-step shift-and-invert Arnoldi process generates an orthonormal basis $\widetilde{V}_{k+1}$ for the Krylov subspace $\mathcal{K}_{k+1}\big((I+\gamma A)^{-1},\widetilde{\bf v}_1\big)={\rm span}\big\{\widetilde{\bf v}_1,(I+\gamma A)^{-1}\widetilde{\bf v}_1,\ldots,[(I+\gamma A)^{-1}]^{k}\widetilde{\bf v}_1\big\}$, and we have the following relation
\begin{equation}\label{equ25}
(I+\gamma A)^{-1}\widetilde{V}_{k}=\widetilde{V}_{k}\widetilde{H}_{k}+\widetilde{h}_{k+1,k}\widetilde{\bf v}_{k+1}{\bf e}_{k}^{\rm H},
\end{equation}
where $\widetilde{H}_{k}$ is a $k$-by-$k$ upper-Hessenberg matrix. If $\widetilde{H}_k$ is nonsingular, we denote
$B_{k}= \frac{\widetilde{H}_{k}^{-1}-I}{\gamma}$, then the shift-and-invert Arnoldi method uses
$$
\widetilde{\bf y}_{k}(t)=\widetilde{V}_k{\rm exp}(-t B_{k})\beta{\bf e}_{1}\equiv \widetilde{V}_{k}\widetilde{\bf u}_{k}(t)
$$
as an approximation to the desired solution, where $\widetilde{\bf u}_{k}(t)={\rm exp}(-t B_{k})\beta{\bf e}_{1}$.

Rewrite the relation (\ref{equ25}) as
\begin{equation}\label{equ26}
A\widetilde{V}_{k}=\widetilde{V}_{k}B_{k}-\frac{\widetilde{h}_{k+1,k}}{\gamma}(I+\gamma A)\widetilde{\bf v}_{k+1}{\bf e}_{k}^{\rm H}\widetilde{H}_{k}^{-1},
\end{equation}
then we have that
$$
A\widetilde{\bf y}_{k}(t)=A\widetilde{V}_{k}\widetilde{\bf u}_{k}(t)=\Big[\widetilde{V}_{k}B_{k}-\frac{\widetilde{h}_{k+1,k}}{\gamma}(I+\gamma A)\widetilde{\bf v}_{k+1}{\bf e}_{k}^{\rm H}\widetilde{H}_{k}^{-1}\Big]{\rm exp}(-t B_{k})\beta{\bf e}_{1},
$$
and
$$
\widetilde{\bf y}_{k}'(t)=-\widetilde{V}_{k}B_{k}{\rm exp}(-t B_{k})\beta{\bf e}_{1}.
$$
So the residual can be expressed as \cite{Residual}
\begin{equation}\label{equ27}
\widetilde{\bf r}_{k}(t)=-A\widetilde{\bf y}_{k}(t)-\widetilde{\bf y}_{k}'(t)=\frac{\widetilde{h}_{k+1,k}}{\gamma}\Big[{\bf e}_{k}^{\rm H}\widetilde{H}_{k}^{-1}{\rm exp}(-t B_{k})\beta{\bf e}_{1}\Big](I+\gamma A)\widetilde{\bf v}_{k+1},
\end{equation}
and
$$
\|\widetilde{\bf r}_{k}(t)\|_2=\Big|\frac{\widetilde{h}_{k+1,k}}{\gamma}\Big[{\bf e}_{k}^{\rm H}\widetilde{H}_{k}^{-1}{\rm exp}(-tB_{k})\beta{\bf e}_{1}\Big]\Big|\cdot\|\widetilde{\bf v}_{k+1}+\gamma A\widetilde{\bf v}_{k+1}\|_2.
$$
It is seen from (\ref{equ27}) that the residual vector $\widetilde{\bf r}_{k}(t)$ is colinear with $(I+\gamma A)\widetilde{\bf v}_{k+1}$, and it is orthogonal to the space ${\rm span}\{(I+\gamma A)^{\rm -H}\widetilde{V}_{k}\}$:
$$
\big[(I+\gamma A)^{\rm -H}\widetilde{V}_{k}\big]^{\rm H}\widetilde{\bf r}_{k}(t)
=\widetilde{V}_{k}^{\rm H}(I+\gamma A)^{-1}\big[-A\widetilde{\bf y}_{k}(t)-\widetilde{\bf y}_{k}'(t)\big]={\bf 0}.
$$
That is,
\begin{equation}\label{eqn27}
\left\{\begin{array}{c}
\widetilde{\bf y}_{k}(t)=\widetilde{V}_k{\rm exp}(-t B_{k})\beta{\bf e}_{1}~\in~{\rm span}\{\widetilde V_{k}\},\\
-A\widetilde{\bf y}_{k}(t)-\widetilde{\bf y}_{k}'(t)~\bot~{\rm span}\{(I+\gamma A)^{\rm -H}\widetilde{V}_{k}\}.
\end{array}\right.
\end{equation}
In other words, the shift-and-invert Arnoldi method for matrix exponential is an oblique projection method.

\subsection{The Arnoldi and the shift-and-invert Arnoldi methods for $\varphi_{\ell}~(\ell\geq 1)$ functions}

We consider the problem of
\begin{equation}\label{eqn111}
{\bf y}(t)=\varphi_{\ell}(-tA){\bf v},\quad \ell=1,2,\ldots,s.
\end{equation}
Given the Arnoldi relation (\ref{equ22}), the Arnoldi method uses \cite{NW1,NW2}
$$
{\bf y}_{\ell,k}(t)=V_{k}\varphi_{\ell}(-tH_{k})\beta{\bf e}_{1}
$$
as an approximate solution to ${\bf y}(t)$. In this subsection, we aim to evaluate the residual of the approximation efficiently, and provide an effective stopping criterion for the computation of (\ref{eqn111}).

Note that ${\bf y}(t)$ solves the following differential equation
\begin{equation}\label{equ49}
\left\{\begin{array}{l} {\bf y}'(t)=-A{\bf y}(t)- \frac{\ell}{t}{\bf y}(t) +\frac{1}{t(\ell-1)!}{\bf v},\\
{\bf y}(0)={\bf v}/\ell!.
\end{array}
\right.
\end{equation}
Then we can define
\begin{equation}\label{eqn2.10}
{\bf r}_{\ell,k}(t)=-A{\bf y}_{\ell,k}(t)-\frac{\ell}{t}{\bf y}_{\ell,k}(t)+\frac{1}{t(\ell-1)!}{\bf v}-{\bf y}_{\ell,k}'(t)
\end{equation}
as a residual of ${\bf y}_{\ell,k}(t)$. On the other hand, we have from the relation
\begin{equation}\label{eqn2.15}
\varphi_{\ell-1}(-tA){\bf v}=-tA\varphi_{\ell}(-tA){\bf v}+\frac{1}{(\ell-1)!}{\bf v}, \quad \ell\geq 1
\end{equation}
that
\begin{equation}\label{eqn2.12}
\breve{{\bf r}}(t)=-t AV_{k}\varphi_{\ell}(-t H_{k})\beta{\bf e}_{1}-V_{k}\varphi_{\ell-1}(-t H_{k})\beta{\bf e}_{1}+\frac{1}{(\ell-1)!}V_{k}\beta{\bf e}_{1}
\end{equation}
can also be utilized as a residual of ${\bf y}_{\ell,k}(t)$.
The following proposition reveals the relationship between the two residuals (\ref{eqn2.10}) and (\ref{eqn2.12}).
\begin{prop}
Under the above notations, we have
\begin{eqnarray}\label{equ51}
{\bf r}_{\ell,k}(t)=\frac{1}{t}\breve{{\bf r}}(t)
=-h_{k+1,k}\big[{\bf e}_{k}^{\rm H}\varphi_{\ell}(-t H_{k})\beta{\bf e}_{1}\big]{\bf v}_{k+1},
\end{eqnarray}
and
\begin{equation}\label{eqn216}
\|{\bf r}_{\ell,k}(t)\|_2=\Big|h_{k+1,k}\big[{\bf e}_{k}^{\rm H}\varphi_{\ell}(-t H_{k})\beta{\bf e}_{1}\big]\Big|.
\end{equation}
\end{prop}
\begin{proof}
It follows from (\ref{equ49}) that
$$
\varphi_{\ell}'(-t H_{k})\beta{\bf e}_{1}=-H_{k}\varphi_{\ell}(-t H_{k})\beta{\bf
e}_{1}-\frac{\ell}{t}\varphi_{\ell}(-t H_{k})\beta{\bf
e}_{1}+\frac{1}{t(\ell-1)!}\beta{\bf
e}_{1}.
$$
Thus,
\begin{eqnarray}\label{equ52}
{\bf r}_{\ell,k}(t)
&=& -A {\bf y}_{\ell,k}(t)-\frac{\ell}{t}{\bf y}_{\ell,k}(t)+\frac{1}{t(\ell-1)!}{\bf v}-{\bf y}_{\ell,k}'(t)\nonumber\\
&=& -AV_{k}\varphi_{\ell}(-t H_{k})\beta{\bf e}_{1}-\frac{\ell}{t}V_{k}\varphi_{\ell}(-t H_{k})\beta{\bf e}_{1}-V_{k}\varphi_{\ell}'(-t H_{k})\beta{\bf e}_{1}+\frac{1}{t(\ell-1)!}{\bf v} \nonumber\\
&=& -AV_{k}\varphi_{\ell}(-t H_{k})\beta{\bf e}_{1}+V_{k}\Big[H_{k}\varphi_{\ell}(-t H_{k})\beta{\bf e}_{1}-\frac{1}{t(\ell-1)!}\beta{\bf e}_{1}\Big]+\frac{1}{t(\ell-1)!}{\bf v}.
\end{eqnarray}
Notice from (\ref{eqn11}) that
$$
\varphi_{\ell-1}(-tH_k)\beta{\bf e}_1=-tH_k\varphi_{\ell}(-tH_k)\beta{\bf e}_1+\frac{1}{(\ell-1)!}\beta{\bf e}_1,\quad \ell=1,2,\ldots
$$
So we have
\begin{equation}\label{wu28}
H_{k}\varphi_{\ell}(-t H_{k})\beta{\bf e}_{1}-\frac{1}{t(\ell-1)!}\beta{\bf e}_{1}=-\frac{1}{t}\varphi_{\ell-1}(-t H_{k})\beta{\bf e}_{1},
\end{equation}
and (\ref{equ52}) can be rewritten as
\begin{equation}\label{equ53}
{\bf r}_{\ell,k}(t)=-AV_{k}\varphi_{\ell}(-t H_{k})\beta{\bf e}_{1}-\frac{1}{t}V_{k}\varphi_{\ell-1}(-t H_{k})\beta{\bf e}_{1}+\frac{1}{t(\ell-1)!}V_{k}\beta{\bf e}_{1}.
\end{equation}
On the other hand,
\begin{eqnarray*}
\breve{{\bf r}}(t)
&=& -t AV_{k}\varphi_{\ell}(-t H_{k})\beta{\bf e}_{1}-V_{k}\varphi_{\ell-1}(-t H_{k})\beta{\bf e}_{1}+\frac{1}{(\ell-1)!}V_{k}\beta{\bf e}_{1}\nonumber\\
&=& t\Big[-AV_{k}\varphi_{\ell}(-t H_{k})\beta{\bf e}_{1}-\frac{1}{t}V_{k}\varphi_{\ell-1}(-t H_{k})\beta{\bf e}_{1}+\frac{1}{t(\ell-1)!}V_{k}\beta{\bf e}_{1}\Big] \nonumber\\
&=& t\cdot {\bf r}_{\ell,k}(t).
\end{eqnarray*}
Moreover, we have from (\ref{equ22}) that
\begin{eqnarray*}
\breve{{\bf r}}(t)
&=& -t A V_{k}\varphi_{\ell}(-t H_{k})\beta{\bf e}_{1}-V_{k}\varphi_{\ell-1}(-t H_{k})\beta{\bf e}_{1}+\frac{1}{(\ell-1)!}V_{k}\beta{\bf e}_{1}\nonumber\\
&=& -t\big(V_{k}H_{k}+h_{k+1,k}{\bf v}_{k+1}{\bf e}_{k}^{\rm H}\big)\varphi_{\ell}(-t H_{k})\beta{\bf e}_{1}-V_{k}\varphi_{\ell-1}(-t H_{k})\beta{\bf e}_{1}+\frac{1}{(\ell-1)!}V_{k}\beta{\bf e}_{1}\nonumber \\
&=& V_{k}\Big[-t H_{k}\varphi_{\ell}(-t H_{k})\beta{\bf e}_{1}-\varphi_{\ell-1}(-t H_{k})\beta{\bf e}_{1}+\frac{1}{(\ell-1)!}\beta{\bf e}_{1}\Big]-t h_{k+1,k}[{\bf e}_{k}^{\rm H}\varphi_{\ell}(-t H_{k})\beta{\bf e}_{1}]{\bf v}_{k+1}\nonumber\\
&=& -t h_{k+1,k}\big[{\bf e}_{k}^{\rm H}\varphi_{\ell}(-t H_{k})\beta{\bf e}_{1}\big]{\bf v}_{k+1},
\end{eqnarray*}
where we used (\ref{wu28}); and (\ref{eqn216}) follows from (\ref{equ51}) and the fact that $\|{\bf v}_{k+1}\|_2=1$.
\end{proof}

\begin{rem}
Proposition 2.1 indicates that the residual vector ${\bf r}_{\ell,k}(t)$ is colinear with the $(k+1)$-th Arnoldi basis vector ${\bf v}_{k+1}$, which is orthogonal to the search space ${\rm span}\{V_{k}\}$, i.e.,
\begin{equation}\label{eqn2.18}
\left\{\begin{array}{c}
{\bf y}_{\ell,k}(t)=V_{k}\varphi_{\ell}(-tH_{k})\beta{\bf e}_{1}~\in~{\rm span}\{V_{k}\},\\
-A{\bf y}_{\ell,k}(t)-\frac{\ell}{t}{\bf y}_{\ell,k}(t)+\frac{1}{t(\ell-1)!}{\bf v}-{\bf y}_{\ell,k}'(t)~\bot~{\rm span}\{V_{k}\}.
\end{array}\right.
\end{equation}
In terms of {\rm(}\ref{eqn2.5}{\rm)} and {\rm(}\ref{eqn2.18}{\rm)}, the Arnoldi method for $\varphi$-functions is an orthogonal projection method.
\end{rem}

Now we consider the shift-and-invert Arnoldi method for $\varphi_{\ell}$ functions with $\ell\geq 1$. Given the shift-and-invert Arnoldi relation
(\ref{equ25}), the shift-and-invert Arnoldi method exploits \cite{Rag}
$$
\widetilde{\bf y}_{\ell,k}(t)=\widetilde{V}_{k}\varphi_{\ell}(-t B_{k})\beta{\bf e}_{1}
$$
as an approximation to ${\bf y}(t)=\varphi_{\ell}(-tA){\bf v}$, where $B_{k}=\frac{\widetilde{H}_{k}^{-1}-I}{\gamma}$. Define
\begin{equation}
\widetilde{{\bf r}}_{\ell,k}(t)=-A\widetilde{\bf y}_{\ell,k}(t)-\frac{\ell}{t}\widetilde{\bf y}_{\ell,k}(t)+\frac{1}{t(\ell-1)!}{\bf v}-\widetilde{\bf y}_{\ell,k}'(t),
\end{equation}
we have the following result on the residual.
\begin{prop}
Under the above notations, there holds
\begin{equation}\label{equ5-5}
\widetilde{{\bf r}}_{\ell,k}(t)
=\frac{\widetilde{h}_{k+1,k}}{\gamma}\Big[{\bf e}_{k}^{\rm H}\widetilde{H}_{k}^{-1}\varphi_{\ell}(-t B_{k})\beta {\bf e}_{1}\Big](I+\gamma A)\widetilde{\bf v}_{k+1},
\end{equation}
and
$$
\|\widetilde{{\bf r}}_{\ell,k}(t)\|_2=\Big|\frac{\widetilde{h}_{k+1,k}}{\gamma}\Big[{\bf e}_{k}^{\rm H}\widetilde{H}_{k}^{-1}\varphi_{\ell}(-t B_{k})\beta {\bf e}_{1}\Big]\Big|\cdot\|\widetilde{\bf v}_{k+1}+\gamma A\widetilde{\bf v}_{k+1}\|_2.
$$
\end{prop}
\begin{proof}
It follows from the shift-and-invert Arnoldi relation and (\ref{eqn11}) that
\begin{eqnarray*}
t\cdot\widetilde{{\bf r}}_{\ell,k}(t)
&=& -tA\widetilde{V}_{k}\varphi_{\ell}(-t B_{k})\beta{\bf e}_{1}-\widetilde{V}_{k}\varphi_{\ell-1}(-t B_{k})\beta{\bf e}_{1}+\frac{1}{(\ell-1)!}\widetilde{V}_{k}\beta{\bf e}_{1}\nonumber\\
&=& -t\big[\widetilde{V}_{k}B_{k}-\frac{\widetilde{h}_{k+1,k}}{\gamma}(I+\gamma A)\widetilde{\bf v}_{k+1}{\bf e}_{k}^{\rm H}\widetilde{H}_{k}^{-1}\big]\varphi_{\ell}(-t B_{k})\beta{\bf e}_{1}-\widetilde{V}_{k}\varphi_{\ell-1}(-t B_{k})\beta{\bf e}_{1}+\frac{1}{(\ell-1)!} \widetilde{V}_{k}\beta{\bf e}_{1}\nonumber \\
&=& \widetilde{V}_{k}\Big[-t B_{k}\varphi_{\ell}(-t B_{k})\beta{\bf e}_{1}-\varphi_{\ell-1}(-t B_{k})\beta{\bf e}_{1}+\frac{1}{(\ell-1)!}\beta{\bf e}_{1}\Big]\\
&&+t\frac{\widetilde{h}_{k+1,k}}{\gamma}\big[{\bf e}_{k}^{\rm H}\widetilde{H}_{k}^{-1}\varphi_{\ell}(-t B_{k})\beta{\bf e}_{1}\big](I+\gamma A)\widetilde{\bf v}_{k+1}\nonumber\\
&=& t\frac{\widetilde{h}_{k+1,k}}{\gamma}\Big[{\bf e}_{k}^{\rm H}\widetilde{H}_{k}^{-1}\varphi_{\ell}(-t B_{k})\beta{\bf e}_{1}\Big](I+\gamma A)\widetilde{\bf v}_{k+1},
\end{eqnarray*}
and
\begin{eqnarray*}
\|\widetilde{{\bf r}}_{\ell,k}(t)\|_2&=&\Big\|\frac{\widetilde{h}_{k+1,k}}{\gamma}\Big[{\bf e}_{k}^{\rm H}\widetilde{H}_{k}^{-1}\varphi_{\ell}(-t B_{k})\beta{\bf e}_{1}\Big](I+\gamma A)\widetilde{\bf v}_{k+1}\Big\|_2\\
&=&\Big|\frac{\widetilde{h}_{k+1,k}}{\gamma}\Big[{\bf e}_{k}^{\rm H}\widetilde{H}_{k}^{-1}\varphi_{\ell}(-t B_{k})\beta {\bf e}_{1}\Big]\Big|\cdot\|\widetilde{\bf v}_{k+1}+\gamma A\widetilde{\bf v}_{k+1}\|_2.
\end{eqnarray*}
\end{proof}

\begin{rem}
Proposition 2.2 shows that
the residual vector $\widetilde{{\bf r}}_{\ell,k}(t)$ is colinear with $(I+\gamma A)\widetilde{\bf v}_{k+1}$, which is orthogonal to ${\rm span}\{(I+\gamma A)^{\rm -H}\widetilde{V}_{k}\}$, i.e.,
\begin{equation}\label{eqn221}
\left\{\begin{array}{c}\widetilde{\bf y}_{\ell,k}(t)=\widetilde{V}_{k}\varphi_{\ell}(-t B_{k})\beta{\bf e}_{1}~\in~ {\rm span}\{\widetilde{V}_{k}\},\\
-A\widetilde{\bf y}_{\ell,k}(t)-\frac{\ell}{t}\widetilde{\bf y}_{\ell,k}(t)+\frac{1}{t(\ell-1)!}{\bf v}-\widetilde{\bf y}_{\ell,k}'(t)~\bot~{\rm span}\{(I+\gamma A)^{\rm -H}\widetilde{V}_{k}\}.
\end{array}\right.
\end{equation}
In view of {\rm(}\ref{eqn27}{\rm)} and {\rm(}\ref{eqn221}{\rm)},
the shift-and-invert Arnoldi method for $\varphi$-functions is an oblique projection method. In this method, however, one has to compute $(I+\gamma A)^{-1}$ in advance, or to solve a shifted linear system in each step of the shift-and-invert Arnoldi process, which is prohibitive for large scale matrices.
\end{rem}

\section{A harmonic Arnoldi method and a thick-restarted harmonic Arnoldi algorithm for $\varphi$-functions}

\setcounter{equation}{0}

In order to accelerate convergence of the standard Arnoldi method, it is preferable to deflate some eigenvalues near a singularity of the function in question \cite{EEG,Frommer,Harmonic}.
For instance, one often needs to deflate some eigenvalues close to a given shift
(e.g., the smallest eigenvalues in magnitude \cite{EEG,Frommer}), so that the convergence speed can be improved significantly. It is well known that the harmonic Arnoldi method is appropriate to interior eigenproblems \cite{Paige,Stewart}.
In \cite{Harmonic}, Hochbruck and Hochstenbach reviewed three different derivations of a harmonic Ritz approach for matrix functions: (i)
using a projection onto the search space; (ii)
approximating the shifted linear systems in the Dunford--Taylor integral representation;
(iii) interpolating the function in certain points.

In this section, we introduce the harmonic Arnoldi method for $\varphi$-functions based on
the residual and the harmonic projection technique for large eigenproblems \cite{Paige,Saad}, and shed light on the relationship between the harmonic Arnoldi method and the standard Arnoldi method. Furthermore, we consider how to restart the harmonic Arnoldi method efficiently, by using the thick-restarting strategy that is popular for large scale eigenproblems and linear equations \cite{JW,Morgan,WS}.
Thanks to the residuals of the harmonic Arnoldi approximations,
our new algorithm can solve (\ref{eqn1.3}) simultaneously in the same search subspace.

\subsection{The harmonic Arnoldi method for matrix exponential}

The Arnoldi relation (\ref{equ22}) can be rewritten as
\begin{equation}\label{equ28}
(I+\gamma A)V_{k}=V_{k}(I+\gamma H_{k})+\gamma h_{k+1,k}{\bf v}_{k+1}{\bf e}_{k}^{\rm H},
\end{equation}
where $\gamma$ is a user-prescribed parameter.
Suppose that we want to seek an approximation $\check{\bf y}_{k}(t)\equiv V_{k}\check{\bf u}_{k}(t)$ to ${\rm exp}\big[-t(I+\gamma A)^{-1}\big)]{\bf v}$ in the subspace ${\rm span}\{V_k\}$, and the residual is \cite{Residual}
$$
\check{\bf r}_{k}(t)=-(I+\gamma A)^{-1}\check{\bf y}_{k}(t)-\check{\bf y}'_{k}(t).
$$

Now we consider how to compute $\check{\bf u}_{k}(t)$. Motivated by the harmonic Arnoldi method for interior eigenvalue problems \cite{Paige,Stewart}, let
$$
\check{\bf r}_{k}(t)\bot {\rm span}\{(I+\gamma A)^{\rm H}(I+\gamma A) V_{k}\},
$$
that is,
\begin{equation}\label{eqn3.1}
V_{k}^{\rm H}(I+\gamma A)^{\rm H}(I+\gamma A)\Big[-(I+\gamma
A)^{-1}V_{k}\check{\bf u}_{k}(t)-V_{k}\check{\bf
u}_{k}'(t)\Big]={\bf 0}.
\end{equation}
If $M_{k}\equiv(I+\gamma H_{k})^{\rm H}(I+\gamma H_{k})+\gamma^{2} h_{k+1,k}^{2} {\bf e}_{k}{\bf e}_{k}^{\rm H}$ is nonsingular, denote $Q_{k}=M_{k}^{-1}(I+\gamma H_{k})^{\rm H}$, then we obtain from (\ref{eqn3.1}) that
$$
\check{\bf u}_{k}'(t)=-Q_{k} \check{\bf u}_{k}(t), \quad \check{\bf u}_{k}(0)=\beta {\bf e}_{1},
$$
and
$$
\check{\bf u}_{k}(t)= {\rm exp}(-t Q_{k})\beta{\bf e}_{1}.
$$
Thus, we can approximate ${\rm exp}\big[-t(I+\gamma A)^{-1}\big)]{\bf v}$ by using
\begin{equation}\label{equa3.2}
\check{\bf y}_{k}(t)=V_{k}\check{\bf u}_{k}(t)= V_{k} {\rm exp}(-t Q_{k})\beta{\bf e}_{1}.
\end{equation}

Denote
\begin{equation}\label{eqn3.3}
T_{k}=\frac{Q_{k}^{-1}-I}{\gamma}=H_{k}+\gamma h_{k+1,k}^{2}(I+\gamma H_{k})^{\rm -H}{\bf e}_{k}{\bf e}_{k}^{\rm H},
\end{equation}
in view of (\ref{equa3.2}), the idea behind the harmonic Arnoldi method is to use
\begin{equation}\label{eqn3.4}
\widehat{\bf y}_{k}(t)=V_{k}\widehat{\bf u}_{k}(t)=V_{k} {\rm exp}(-tT_{k})\beta{\bf e}_{1}
\end{equation}
as an approximation to ${\bf y}(t)={\rm exp}(-tA){\bf v}$, where $\widehat{\bf u}_{k}(t)={\rm exp}(-tT_{k})\beta{\bf e}_{1}$. The residual is
\begin{eqnarray}\label{eqn3.5}
\widehat{\bf r}_{k}(t)
&=& -A\widehat{\bf y}_{k}(t)-\widehat{\bf y}_{k}'(t) \nonumber\\
&=& \Big[\gamma h_{k+1,k}^{2}V_{k}(I+\gamma H_{k})^{\rm -H}{\bf e}_{k}{\bf e}_{k}^{\rm H}-h_{k+1,k}{\bf v}_{k+1}{\bf e}_{k}^{\rm H}\Big]{\rm exp}(-t T_{k})\beta {\bf e}_{1}\label{eqn3.6} \\
&=& V_{k+1}\left[\begin{array}{c} \gamma h_{k+1,k}^{2}(I+\gamma H_{k})^{\rm -H}{\bf e}_{k}\nonumber\\
-h_{k+1,k}
\end{array}\right]\Big[{\bf e}_{k}^{\rm H}{\rm exp}(-t T_{k})\beta {\bf e}_{1}\Big],
\end{eqnarray}
and
%\begin{equation}\label{eqn3.7}
$$
\|\widehat{\bf r}_{k}(t)\|_2=\Bigg\| \left[\begin{array}{c}  \gamma h_{k+1,k}^{2}(I+\gamma H_{k})^{-\rm H}{\bf e}_{k}\\
-h_{k+1,k}
\end{array}\right]\Bigg\|_2 \cdot \big|{\bf e}_{k}^{\rm H}{\rm exp}(-t T_{k})\beta {\bf e}_{1}\big|.
$$
%\end{equation}
Thus, the residual vector $\widehat{\bf r}_{k}(t)$ is colinear with the residual of the harmonic Ritz pairs \cite{Paige,Stewart}, and
the harmonic Arnoldi method for matrix exponential is an oblique projection method in the sense that
\begin{equation}\label{eqn377}
\left\{\begin{array}{c}
\widehat{\bf y}_{k}(t)=V_{k}{\rm exp}(-tT_{k})\beta{\bf e}_{1}~\in~{\rm span}\{V_{k}\},\\
-A\widehat{\bf y}_{k}(t)-\widehat{\bf y}_{k}'(t)~\bot~{\rm span}\{(I+\gamma A)V_{k}\}.
\end{array}\right.
\end{equation}
The method is an alternative Krylov subspace approach to the Arnoldi method. Indeed, the Arnoldi approximation ${\bf y}_k(t)$ can be
characterized as ${\bf y}_k(t)=q_{H_k}(A){\bf v}$, where
$q_{H_k}$ interpolates the exponential function at the eigenvalues of $H_k$, which are the Ritz values
of $A$ with respect to the Krylov subspace $\mathcal{K}_k(A,{\bf v})$ \cite{Saad}. Alternatively, the harmonic Arnoldi approximation $\widehat{\bf y}_k(t)$
is based on polynomial interpolation at the harmonic Ritz values instead
of the standard Ritz values \cite{Frommer}.

To our best knowledge, however, the relationship between the harmonic Arnoldi method and the standard Arnoldi method is still unknown.
The following theorem establishes the relationship of the approximations and residuals of the two approaches for matrix exponential.
\begin{theorem}
Denote by ${\bf y}_{k}(t)=V_{k} {\rm exp}(-t H_{k})\beta{\bf e}_{1}$ and $ \widehat{\bf y}_{k}(t)=V_{k}{\rm exp}(-t T_{k})\beta{\bf e}_{1}$ the Arnoldi approximation and the harmonic Arnoldi approximation to ${\bf y}(t)$, respectively; and by ${\bf r}_{k}(t)$ and $\widehat{\bf r}_{k}(t)$ the corresponding residuals. Let
${\bf g}=\gamma h_{k+1,k}^{2}(I+\gamma H_{k})^{\rm -H}{\bf e}_{k}$ and $k\geq 2$,
then we have
\begin{equation}
\widehat{\bf y}_{k}(t)={\bf y}_{k}(t)+
\beta V_{k}\sum _{m=2}^{\infty} \bigg(\Big[\frac{(-t)^{m}}{m!}\Big]\sum _{p=1}^{m-1} \big[{\bf e}_{k}^{\rm H}(T_{k})^{p}{\bf e}_{1}\big]\big[(H_{k})^{m-1-p}{\bf g}\big]\bigg),
\end{equation}
and
\begin{equation}\label{eqn3.9}
\widehat{\bf r}_{k}(t)={\bf r}_{k}(t)-(h_{k+1,k}\alpha) {\bf v}_{k+1}+\gamma h_{k+1,k}^{2}\Big[{\bf e}_{k}^{\rm H}{\rm
exp}(-tH_{k})\beta {\bf e}_{1}+ \alpha \Big] V_{k} (I+\gamma
H_{k})^{\rm -H} {\bf e}_{k},
\end{equation}
where
\begin{equation*}
\alpha=
\beta \sum _{m=2}^{\infty}
\Big(\Big[\frac{(-t)^{m}}{m!}\Big]\sum _{p=1}^{m-1} \big[{\bf
e}_{k}^{\rm H}(T_{k})^{p}{\bf e}_{1}\big]\big[{\bf e}_{k}^{\rm H}(H_{k})^{m-1-p}{\bf g}\big]\Big).
\end{equation*}
\end{theorem}
\begin{proof}
By the definition of matrix exponential, we have
$$
{\bf u}_{k}(t)={\rm exp}(-t H_{k})\beta{\bf e}_{1}=\beta  \sum _{m=0}^{\infty}\Big[\frac{(-t H_{k})^{m}}{m!}\Big]{\bf e}_{1},
$$
$$
\widehat{\bf u}_{k}(t)={\rm exp}(-t T_{k})\beta{\bf e}_{1}=\beta \sum _{m=0}^{\infty}\Big[\frac{(-t T_{k})^{m}}{m!}\Big]{\bf e}_{1},
$$
and
\begin{equation}\label{equ32}
\widehat{\bf u}_{k}(t)-{\bf u}_{k}(t)=\beta \sum _{m=0}^{\infty} \Big[\frac{(-t)^{m}}{m!}\Big]\Big[(T_{k})^{m}{\bf e}_{1}-(H_{k})^{m}{\bf e}_{1}\Big].
\end{equation}
Furthermore, for $k\geq 2$, we have from $T_{k}{\bf e}_{1}=\big[H_{k}+\gamma h_{k+1,k}^{2}(I+\gamma
H_{k})^{\rm -H}{\bf e}_{k}{\bf e}_{k}^{\rm H}\big]{\bf e}_{1}=H_{k}{\bf e}_{1}$ that
\begin{eqnarray*}
\sum _{m=0}^{\infty} \Big[\frac{(-t)^{m}}{m!}\Big]\Big[(T_{k})^{m}{\bf e}_{1}-(H_{k})^{m}{\bf e}_{1}\Big]&=&\sum _{m=1}^{\infty} \Big[\frac{(-t)^{m}}{m!}\Big]\Big[(T_{k})^{m}{\bf e}_{1}-(H_{k})^{m}{\bf e}_{1}\Big]\\
&=&\sum _{m=2}^{\infty} \Big[\frac{(-t)^{m}}{m!}\Big]\Big[(T_{k})^{m}{\bf e}_{1}-(H_{k})^{m}{\bf e}_{1}\Big].
\end{eqnarray*}
Thus,
\begin{eqnarray}\label{equ3-3}
\widehat{\bf y}_{k}(t)-{\bf y}_{k}(t)
&=&V_{k}\widehat{\bf u}_{k}(t)-V_{k}{\bf u}_{k}(t) \nonumber\\
&=&\beta V_{k} \sum _{m=2}^{\infty} \Big[\frac{(-t)^{m}}{m!}\Big]\Big[(T_{k})^{m}{\bf e}_{1}-(H_{k})^{m}{\bf e}_{1}\Big].
\end{eqnarray}
%In the following we use induction to get the relation of $(T_{k})^{m}{\bf e}_{1}-(H_{k})^{m}{\bf e}_{1}$.
For notation simplicity, we denote
$$
T_{k}=H_{k}+\gamma h_{k+1,k}^{2}(I+\gamma
H_{k})^{\rm -H}{\bf e}_{k}{\bf e}_{k}^{\rm H}\equiv H_{k}+{\bf g}{\bf
e}_{k}^{\rm H},
$$
where ${\bf g}=\gamma h_{k+1,k}^{2}(I+\gamma H_{k})^{\rm -H}{\bf
e}_{k}$.
Note that $T_k$ is a rank-one update of $H_k$.
Next we verify that
\begin{equation}\label{36}
(T_{k})^{m}{\bf
e}_{1}-(H_{k})^{m}{\bf e}_{1}= \sum _{p=1}^{m-1} \big[{\bf
e}_{k}^{\rm H}(T_{k})^{p}{\bf e}_{1}\big]\big[(H_{k})^{m-1-p}{\bf g}\big].
\end{equation}
Indeed, we note that
$$
(T_{k})^{2}{\bf e}_{1}-(H_{k})^{2}{\bf e}_{1}
=[{\bf e}_{k}^{\rm H}T_{k}{\bf e}_{1}](H_{k})^{0}{\bf g}.
$$
Assume that
\begin{equation}\label{equ33}
(T_{k})^{j}{\bf e}_{1}-(H_{k})^{j}{\bf e}_{1}= \sum _{p=1}^{j-1} \big[{\bf e}_{k}^{\rm H}(T_{k})^{p}{\bf
e}_{1}\big]\big[(H_{k})^{j-1-p}{\bf g}\big],\quad j\geq 2.
\end{equation}
Thus,
\begin{equation*}
T_{k}\Big[(T_{k})^{j}{\bf e}_{1}-(H_{k})^{j}{\bf e}_{1}\Big]
=(T_{k})^{j+1}{\bf e}_{1}-(H_{k})^{j+1}{\bf e}_{1}-\Big[{\bf e}_{k}^{\rm H}(H_{k})^{j}{\bf e}_{1}\Big]{\bf g}.
\end{equation*}
By (\ref{equ33}), we obtain
$$
T_{k}\Big[(T_{k})^{j}{\bf e}_{1}-(H_{k})^{j}{\bf e}_{1}\Big]
=T_{k} \sum _{p=1}^{j-1} \big[{\bf e}_{k}^{\rm H}(T_{k})^{p}{\bf
e}_{1}\big]\big[(H_{k})^{j-1-p}{\bf g}\big],
$$
and
\begin{equation}\label{equ35}
(T_{k})^{j+1}{\bf e}_{1}-(H_{k})^{j+1}{\bf e}_{1} =\sum _{p=1}^{j-1}
\Big[{\bf e}_{k}^{\rm H}(T_{k})^{p}{\bf e}_{1}\Big]
\Big[(H_{k})^{j-p}{\bf g}\Big]+ \Big[ \sum _{p=1}^{j-1} \big[{\bf
e}_{k}^{\rm H}(T_{k})^{p}{\bf e}_{1}\big]\big[{\bf
e}_{k}^{\rm H}(H_{k})^{j-1-p}{\bf g}\big]+{\bf e}_{k}^{\rm H}(H_{k})^{j}{\bf
e}_{1} \Big]{\bf g}.
\end{equation}
Moreover, it follows from (\ref{equ33}) that
$$
{\bf e}_{k}^{\rm H}(T_{k})^{j}{\bf e}_{1}= \sum _{p=1}^{j-1} \big[{\bf
e}_{k}^{\rm H}(T_{k})^{p}{\bf e}_{1}\big]\big[{\bf
e}_{k}^{\rm H}(H_{k})^{j-1-p}{\bf g}\big]+{\bf e}_{k}^{\rm H}(H_{k})^{j}{\bf
e}_{1},
$$
so we arrive at
\begin{eqnarray*}
(T_{k})^{j+1}{\bf e}_{1}-(H_{k})^{j+1}{\bf e}_{1}
&=& \sum _{p=1}^{j-1}\big[{\bf e}_{k}^{\rm H}(T_{k})^{p}{\bf e}_{1}\big]\big[(H_{k})^{j-p}{\bf g}\big]+\big[{\bf e}_{k}^{\rm H}(T_{k})^{j}{\bf e}_{1}\big]{\bf g}\nonumber\\
&=& \sum _{p=1}^{j}
\big[{\bf e}_{k}^{\rm H}(T_{k})^{p}{\bf e}_{1}\big]\big[(H_{k})^{j-p}{\bf g}\big],
\end{eqnarray*}
and the relation (\ref{36}) is established.
Combining (\ref{equ3-3}) and (\ref{36}) yields
\begin{eqnarray*}
\widehat{\bf y}_{k}(t)-{\bf y}_{k}(t)
&=& \beta V_{k} \sum _{m=2}^{\infty}\Big[\frac{(-t)^{m}}{m!}\Big] \Big[(T_{k})^{m}{\bf e}_{1}-(H_{k})^{m}{\bf e}_{1}\Big]\nonumber \\
&=& \beta V_{k} \sum _{m=2}^{\infty}\Big[\frac{(-t)^{m}}{m!}\Big] \Big(\sum _{p=1}^{m-1} \big[{\bf e}_{k}^{\rm H}(T_{k})^{p}{\bf e}_{1}\big]\big[(H_{k})^{m-1-p}{\bf g}\big]\Big).
\end{eqnarray*}
In order to prove (\ref{eqn3.9}), denote
\begin{equation*}
\alpha={\bf e}_{k}^{\rm H}\widehat{\bf u}_{k}(t)-{\bf e}_{k}^{\rm H}{\bf u}_{k}(t)
=\beta \sum _{m=2}^{\infty}
\Big[\frac{(-t)^{m}}{m!}\Big] \sum _{p=1}^{m-1} \big[{\bf
e}_{k}^{\rm H}(T_{k})^{p}{\bf e}_{1}\big]\big[{\bf e}_{k}^{\rm H}(H_{k})^{m-1-p}{\bf g}\big],
\end{equation*}
and we derive from (\ref{eqn3.6}) that
\begin{eqnarray*}
\widehat{\bf r}_{k}(t)
&=&-h_{k+1,k}\big[{\bf e}_{k}^{\rm H}{\bf u}_{k}(t) + \alpha \big]  {\bf v}_{k+1}+\gamma h_{k+1,k}^{2}[{\bf e}_{k}^{\rm H}{\bf u}_{k}(t) + \alpha ] V_{k} (I+\gamma
H_{k})^{\rm -H} {\bf e}_{k}\nonumber\\
&=&{\bf r}_{k}(t)-(h_{k+1,k}\alpha){\bf v}_{k+1}+\gamma h_{k+1,k}^{2}\big[{\bf e}_{k}^{\rm H}{\rm exp}(-t H_{k})\beta{\bf e}_{1}+ \alpha \big] V_{k} (I+\gamma
H_{k})^{\rm -H} {\bf e}_{k}.
\end{eqnarray*}
\end{proof}

%\begin{rem}
%%We point out that the harmonic Arnoldi approximation is not new. Indeed, it is the harmonic Ritz approach proposed in \cite{Harmonic}.
%In \cite{Harmonic}, Hochbruck and Hochstenbach reviewed three different derivations of a harmonic Ritz approach for matrix functions.
%However, they did not consider the residual of matrix function.
%As a comparison, our derivation is based on the residual of the matrix exponential and the harmonic projection technique for large eigenproblems, so it can be viewed as a complement to the harmonic extraction proposed in \cite{Harmonic}.
%\end{rem}

\subsection{A harmonic Arnoldi method for $\varphi_{\ell}~(\ell\geq 1)$ functions}

In this subsection, we focus on the harmonic Arnoldi method for
%\begin{equation}\label{equ11}
$$
{\bf y}(t)=\varphi_{\ell}(-tA){\bf v},\quad \ell=1,2,\ldots,s,
$$
%\end{equation}
and establish the relationship between the harmonic Arnoldi approximation and the Arnoldi approximation for $\varphi_{\ell}$ functions with $\ell\geq 1$.

Given the Arnoldi relation (\ref{equ22}), in the harmonic Arnoldi method, we make use of
\begin{equation}\label{eqnn3.18}
\widehat{\bf y}_{\ell,k}(t)=V_{k}\varphi_{\ell}(-t T_{k})\beta{\bf e}_{1}\equiv V_{k}\widehat{{\bf
u}}_{\ell,k}(t)
\end{equation}
as an approximate solution to ${\bf y}(t)$ in the Krylov subspace $\mathcal{K}_k(A,{\bf v})$, where $T_k$ is defined in (\ref{eqn3.3}).
Denote by
\begin{equation}\label{equ50}
\widehat{\bf r}_{\ell,k}(t)=-A\widehat{\bf y}_{\ell,k}(t)-\frac{\ell}{t}\widehat{\bf y}_{\ell,k}(t)+\frac{1}{t(\ell-1)!}{\bf v}-\widehat{\bf y}_{\ell,k}'(t)
\end{equation}
the residual with respect to $\widehat{\bf y}_{\ell,k}(t)$, we have the following result:
\begin{prop}
Under the above notations, we have
\begin{eqnarray}
\widehat{\bf r}_{\ell,k}(t)&=&\frac{1}{t}\big[-t AV_{k}\varphi_{\ell}(-t T_{k})\beta{\bf e}_{1}-V_{k}\varphi_{\ell-1}(-t T_{k})\beta{\bf e}_{1}+\frac{1}{(\ell-1)!}V_{k}\beta{\bf e}_{1}\big]\label{eqn3.33}\\
&=&V_{k+1} \left[\begin{array}{c} \gamma h_{k+1,k}^{2}(I+\gamma H_{k})^{\rm -H}{\bf e}_{k}\\
-h_{k+1,k}
\end{array}\right]\big[{\bf e}_{k}^{\rm H}\varphi_{\ell}(-t T_{k})\beta{\bf e}_{1}\big]\label{eqnn3.34},
\end{eqnarray}
and
\begin{equation}\label{eqn3.35}
[(I+\gamma A)V_{k}]^{\rm H}\cdot \widehat{{\bf r}}_{\ell,k}(t)={\bf 0}.
\end{equation}
\end{prop}
\begin{proof}
The proof of (\ref{eqn3.33}) is similar to that of (\ref{equ51}), and is omitted.
For (\ref{eqnn3.34}), it follows from (\ref{equ22}), (\ref{eqn2.15}) and (\ref{eqn3.3}) that
\begin{eqnarray*}\label{equ56}
&& -t A V_{k}\varphi_{\ell}(-t T_{k})\beta{\bf e}_{1}-V_{k} \varphi_{\ell-1}(-t T_{k})\beta{\bf e}_{1}+\frac{1}{(\ell-1)!}V_{k}\beta{\bf e}_{1}\nonumber \\
&=& -t\big(V_{k}H_{k}+h_{k+1,k}{\bf v}_{k+1}{\bf e}_{k}^{\rm H}\big)\varphi_{\ell}(-t T_{k})\beta{\bf e}_{1}-V_{k}\varphi_{\ell-1}(-t T_{k})\beta{\bf e}_{1}+\frac{1}{(\ell-1)!}V_{k}\beta{\bf e}_{1}\nonumber  \\
&=& t\gamma h_{k+1,k}^{2}\big[{\bf e}_{k}^{\rm H}\varphi_{\ell}(-t T_{k})\beta{\bf e}_{1}\big] V_{k}(I+\gamma H_{k})^{\rm -H}{\bf e}_{k}-t h_{k+1,k}\big[{\bf e}_{k}^{\rm H}\varphi_{\ell}(-t T_{k})\beta{\bf e}_{1}\big]{\bf v}_{k+1}\nonumber \\
&=& V_{k+1} \left[\begin{array}{c} t\gamma h_{k+1,k}^{2}(I+\gamma H_{k})^{\rm -H}{\bf e}_{k}\\
-t h_{k+1,k}
\end{array}\right] \big[{\bf e}_{k}^{\rm H}\varphi_{\ell}(-t T_{k})\beta{\bf e}_{1}\big],
\end{eqnarray*}
and the relation (\ref{eqn3.35}) is from (\ref{eqnn3.34}) and (\ref{equ28}).
\end{proof}

\begin{rem}
Proposition 3.1 indicates that the residual vector $\widehat{\bf r}_{\ell,k}(t)$ with $\ell\geq 1$ is orthogonal to the space ${\rm span}\{(I+\gamma A)V_{k}\}$, i.e.,
\begin{equation}\label{eqn3.20}
\left\{\begin{array}{c}
\widehat{\bf y}_{\ell,k}(t)=V_{k}\varphi_{\ell}(-t T_{k})\beta{\bf e}_{1}~\in~{\rm span}\{V_{k}\},\\
-A\widehat{\bf y}_{\ell,k}(t)-\frac{\ell}{t}\widehat{\bf y}_{\ell,k}(t)+\frac{1}{t(\ell-1)!}{\bf v}-\widehat{\bf y}_{\ell,k}'(t)~\bot~{\rm span}\{(I+\gamma A) V_{k}\}.
\end{array}\right.
\end{equation}
From {\rm(}\ref{eqn377}{\rm)} and {\rm(}\ref{eqn3.20}{\rm)}, the residual vectors $\widehat{\bf r}_{\ell,k}(t)~(\ell\geq 0)$ are colinear with the residual of the harmonic Ritz pairs \cite{Paige,Stewart}.
\end{rem}

The following theorem establishes the relationship between the Arnoldi method and the harmonic Arnoldi method for $\varphi_{\ell}$ functions with $\ell\geq 1$.
\begin{theorem}
Let ${\bf y}_{\ell,k}(t)=V_{k}\varphi_{\ell}(-t H_{k})\beta{\bf e}_{1}$ and $ \widehat{{\bf y}}_{\ell,k}(t)=V_{k}\varphi_{\ell}(-t T_{k})\beta{\bf e}_{1}$ be the approximate solutions obtained from the Arnoldi method and the harmonic Arnoldi method for the $\varphi_{\ell}~(\ell\geq 1)$ functions, respectively; and denote by ${\bf r}_{\ell,k}(t)$ and $\widehat{{\bf r}}_{\ell,k}(t)$ the corresponding residuals defined in (\ref{eqn2.10}) and (\ref{equ50}). If $H_k$ is nonsingular and $k\geq 2$, then we have that
\begin{eqnarray}
\widehat{{\bf y}}_{\ell,k}(t)-{\bf y}_{\ell,k}(t)
&=&V_{k}\Big[(-t H_{k})^{-\ell}\beta\sum _{m=2}^{\infty}\frac{(-t)^{m}}{m!}\sum _{p=1}^{m-1}[{\bf e}_{k}^{\rm H}(T_{k})^{p}{\bf e}_{1}][(H_{k})^{m-1-p}{\bf g}]\nonumber\\
&&-\sum _{j=1}^{\ell}[{\bf e}_{k}^{\rm H}\varphi_{j}(-t T_{k})\beta{\bf e}_{1}](-t H_{k})^{-\ell+j}[(H_{k})^{-1}{\bf g}]\Big],
\end{eqnarray}
and
\begin{equation}
\widehat{\bf r}_{\ell,k}(t)={\bf r}_{\ell,k}(t)-(h_{k+1,k}\varrho) {\bf v}_{k+1}+\gamma h_{k+1,k}^{2}\big[{\bf e}_{k}^{\rm H}\varphi_{\ell}(-t H_{k})\beta{\bf e}_{1}+\varrho \big] V_{k}(I+\gamma H_{k})^{\rm -H}{\bf e}_{k},
\end{equation}
where
\begin{equation*}
\varrho=\frac{\beta}{(\ell-1)!}{\int_0^1 \sum _{m=2}^{\infty}\Big[\frac{t^{m}(\theta-1)^{m}}{m!}\Big]\sum _{p=1}^{m-1}\big[{\bf e}_{k}^{\rm H}(T_{k})^{p}{\bf e}_{1}\big]\big[{\bf e}_{k}^{\rm H}(H_{k})^{m-1-p}{\bf g}\big]\theta^{\ell-1}d\theta}.
\end{equation*}
\end{theorem}
\begin{proof}
It follows from the recurrence relation of the $\varphi$-functions that
$$
\varphi_{\ell-1}(-t T_{k})\beta{\bf e}_{1}=-t T_{k}\varphi_{\ell}(-t T_{k})\beta{\bf e}_{1}+\frac{1}{(\ell-1)!}\beta{\bf e}_{1},\quad \ell=1,2,\ldots
$$
and
$$
\varphi_{\ell-1}(-t H_{k})\beta{\bf e}_{1}=-t H_{k}\varphi_{\ell}(-t H_{k})\beta{\bf e}_{1}+\frac{1}{(\ell-1)!}\beta{\bf e}_{1},\quad \ell=1,2,\ldots
$$
Therefore,
\begin{equation}\label{equ80}
\varphi_{\ell}(-t T_{k})\beta{\bf e}_{1}-\varphi_{\ell}(-t H_{k})\beta{\bf e}_{1}=(-t H_{k})^{-1}\Big[\varphi_{\ell-1}(-t T_{k})\beta{\bf e}_{1}-\varphi_{\ell-1}(-t H_{k})\beta{\bf e}_{1}\Big]-[(H_{k})^{-1}{\bf g}]\big[{\bf e}_{k}^{\rm H}\varphi_{\ell}(-t T_{k})\beta{\bf e}_{1}\big].
\end{equation}
Using the same trick, we obtain
\begin{eqnarray*}
\varphi_{\ell-1}(-t T_{k})\beta{\bf e}_{1}-\varphi_{\ell-1}(-t H_{k})\beta{\bf e}_{1}&=&(-t H_{k})^{-1}\Big[\varphi_{\ell-2}(-t T_{k})\beta{\bf e}_{1}-\varphi_{\ell-2}(-t H_{k})\beta{\bf e}_{1}\Big]\\
&&-[(H_{k})^{-1}{\bf g}]\big[{\bf e}_{k}^{\rm H}\varphi_{\ell-1}(-t T_{k})\beta{\bf e}_{1}\big].
\end{eqnarray*}
As a result, the relation (\ref{equ80}) can be written as
\begin{eqnarray*}
\varphi_{\ell}(-t T_{k})\beta{\bf e}_{1}-\varphi_{\ell}(-t H_{k})\beta{\bf e}_{1}&=&(-t H_{k})^{-1}\Bigg[(-t H_{k})^{-1}\Big[\varphi_{\ell-2}(-t T_{k})\beta{\bf e}_{1}-\varphi_{\ell-2}(-t H_{k})\beta{\bf e}_{1}\Big]\nonumber\\
&&-[(H_{k})^{-1}{\bf g}]\big[{\bf e}_{k}^{\rm H}\varphi_{\ell-1}(-t T_{k})\beta{\bf e}_{1}\big]\Bigg]-[(H_{k})^{-1}{\bf g}]\big[{\bf e}_{k}^{\rm H}\varphi_{\ell}(-t T_{k})\beta{\bf e}_{1}\big]\nonumber\\
&=& (-t H_{k})^{-2}\Big[\varphi_{\ell-2}(-t T_{k})\beta{\bf e}_{1}-\varphi_{\ell-2}(-t H_{k})\beta{\bf e}_{1}\Big]\nonumber\\
&& -(-t H_{k})^{-1}[(H_{k})^{-1}{\bf g}]\big[{\bf e}_{k}^{\rm H}\varphi_{\ell-1}(-t T_{k})\beta{\bf e}_{1}\big]-[(H_{k})^{-1}{\bf g}]\big[{\bf e}_{k}^{\rm H}\varphi_{\ell}(-t T_{k})\beta{\bf e}_{1}\big].
\end{eqnarray*}
By induction,
\begin{eqnarray*}
\varphi_{\ell}(-t T_{k})\beta{\bf e}_{1}-\varphi_{\ell}(-t H_{k})\beta{\bf e}_{1}&=&(-t H_{k})^{-\ell}\Big[\varphi_{0}(-t T_{k})\beta{\bf e}_{1}-\varphi_{0}(-t H_{k})\beta{\bf e}_{1}\Big]\nonumber\\
&&-\sum _{j=1}^{\ell}\big[{\bf e}_{k}^{\rm H}\varphi_{j}(-t T_{k})\beta{\bf e}_{1}\big](-t H_{k})^{-\ell+j}\big[(H_{k})^{-1}{\bf g}\big].
\end{eqnarray*}
From (\ref{36}), we obtain
\begin{eqnarray}
\varphi_{\ell}(-t T_{k})\beta{\bf e}_{1}-\varphi_{\ell}(-t H_{k})\beta{\bf e}_{1}&=&(-t H_{k})^{-\ell}\beta\sum _{m=2}^{\infty}\frac{(-t)^{m}}{m!}\sum _{p=1}^{m-1}\big[{\bf e}_{k}^{\rm H}(T_{k})^{p}{\bf e}_{1}\big]\big[(H_{k})^{m-1-p}{\bf g}\big]\nonumber\\
&&-\sum _{j=1}^{\ell}\big[{\bf e}_{k}^{\rm H}\varphi_{j}(-t T_{k})\beta{\bf e}_{1}\big](-t H_{k})^{-\ell+j}\big[(H_{k})^{-1}{\bf g}\big].
\end{eqnarray}
So we get
\begin{eqnarray}
\widehat{{\bf y}}_{\ell,k}(t)-{\bf y}_{\ell,k}(t)
&=&V_{k}\big[\varphi_{\ell}(-t T_{k})\beta{\bf e}_{1}-\varphi_{\ell}(-t H_{k})\beta{\bf e}_{1}\big]\nonumber\\
&=&V_{k}\cdot\Big[(-t H_{k})^{-\ell}\beta\sum _{m=2}^{\infty}\frac{(-t)^{m}}{m!}\sum _{p=1}^{m-1}[{\bf e}_{k}^{\rm H}(T_{k})^{p}{\bf e}_{1}][(H_{k})^{m-1-p}{\bf g}]\nonumber\\
&&-\sum _{j=1}^{\ell}[{\bf e}_{k}^{\rm H}\varphi_{j}(-t T_{k})\beta{\bf e}_{1}](-t H_{k})^{-\ell+j}[(H_{k})^{-1}{\bf g}]\Big].
\end{eqnarray}

On the other hand, we have that
\begin{equation*}
\varphi_{\ell}(-t T_{k})\beta{\bf e}_{1}=\frac{1}{(\ell-1)!}{\int_0^1 {\rm exp}\big[(1-\theta)(-t T_{k})\big]\theta^{\ell-1}\beta{\bf e}_{1}d\theta},
\end{equation*}
\begin{equation*}
\varphi_{\ell}(-t H_{k})\beta{\bf e}_{1}=\frac{1}{(\ell-1)!}{\int_0^1 {\rm exp}\big[(1-\theta)(-t H_{k})\big]\theta^{\ell-1}\beta{\bf e}_{1}d\theta},
\end{equation*}
and we have from (\ref{36}) that
\begin{eqnarray}\label{equ59}
\varrho
&=&{\bf e}_{k}^{\rm H}\varphi_{\ell}(-t T_{k})\beta{\bf e}_{1}-{\bf e}_{k}^{\rm H}\varphi_{\ell}(-t H_{k})\beta{\bf e}_{1}\nonumber\\
&=&\frac{\beta}{(\ell-1)!}{\int_0^1 \sum _{m=2}^{\infty}\Big[\frac{t^{m}(\theta-1)^{m}}{m!}\Big]\sum _{p=1}^{m-1}\big[{\bf e}_{k}^{\rm H}(T_{k})^{p}{\bf e}_{1}\big]\big[{\bf e}_{k}^{\rm H}(H_{k})^{m-1-p}{\bf g}\big]\theta^{\ell-1}d\theta}.
\end{eqnarray}
Then we obtain from (\ref{eqnn3.34}) that
\begin{eqnarray*}
\widehat{\bf r}_{\ell,k}(t)
&=& \gamma h_{k+1,k}^{2}\big[{\bf e}_{k}^{\rm H}\varphi_{\ell}(-t H_{k})\beta{\bf e}_{1}+ \varrho \big] V_{k}(I+\gamma H_{k})^{\rm -H}{\bf e}_{k}-h_{k +1,k}\big[{\bf e}_{k}^{\rm H}\varphi_{\ell}(-t H_{k})\beta{\bf e}_{1}+\varrho \big]{\bf v}_{k+1}\nonumber\\
&=&{\bf r}_{\ell,k}(t)-(h_{k+1,k} \varrho) {\bf v}_{k+1}+\gamma h_{k+1,k}^{2}\big[{\bf e}_{k}^{\rm H}\varphi_{\ell}(-t H_{k})\beta{\bf e}_{1}+\varrho \big] V_{k}(I+\gamma H_{k})^{\rm -H}{\bf e}_{k}.
\end{eqnarray*}
\end{proof}

\subsection{A thick-restarted harmonic Arnoldi algorithm for $\varphi$-functions}

When using the Krylov subspace method for approximating the action of a matrix
function on a vector, the maximum number of iterations that can be performed is often limited by
the storage requirements of the full Arnoldi basis.
In this subsection, we propose a thick-restarted harmonic Arnoldi algorithm for the $\varphi$-functions.
The thick-restarting strategy was firstly proposed by Wu and Simon
for large symmetric eigenvalue problem \cite{WS}, and was generalized to solving large non-Hermitian eigenproblems \cite{JW,MZ} and linear systems \cite{Morgan}.

Our thick-restarted harmonic Arnoldi algorithm is a little similar
to the deflated GMRES algorithm for linear systems \cite{Morgan}. So it is simple to implement compared with the deflated Krylov subspace methods proposed in \cite{EEG,From}.
The key to our new algorithm is two-fold. First, we apply an additive correction to the {\it residual}
of the $\varphi$-functions when restarting, instead of the {\it error} used in \cite{ETNA,MA,SINUM,EEG,Frommer,IMA,SISC}. Second, we use the fact that the residual of the harmonic Arnoldi approximation is colinear with that of the harmonic Ritz pairs. Consequently, one can compute the approximations to (\ref{eqn1.3}) simultaneously in the same search subspace.

Let's consider how to thick-restart the harmonic Arnoldi
method for the $\varphi$-functions. We denote $\widehat{\bf y}_{0,k}(t)=\widehat{\bf y}_k(t)$ and $\widehat{\bf r}_{0,k}(t)=\widehat{\bf r}_k(t)$. In the first cycle of the thick-restarted harmonic Arnoldi algorithm, we run the $k$-step Arnoldi process and generate the Arnoldi relation (\ref{equ22}).
We then compute the approximations $\widehat{\bf y}_{0,k}(t)$ and $\widehat{\bf y}_{\ell,k}(t)$ via (\ref{eqn3.4}) and (\ref{eqnn3.18}), respectively. If the residual norms are larger than a given tolerance {\it tol} \big(see (\ref{eqn3.5}) and (\ref{eqnn3.34})\big), one computes some harmonic Ritz pairs $(\widetilde{\lambda}_i,\widetilde{\bf x}_i)~(i=1,2,\ldots,q)$ of $A$, which satisfy \cite{Paige,Stewart}
\begin{equation}
\bigg\{\begin{array}{c}
\widetilde{\bf x}_{i}\in {\rm span}\{V_k^{(1)}\},\\
A\widetilde{\bf x}_{i}-\widetilde{\lambda}_{i}\widetilde{\bf x}_{i}\bot (I+\gamma A){\rm span}\{V_k^{(1)}\}.
\end{array}
\end{equation}
For simplicity, we denote by the variables computed from the ``previous" cycle with a superscript $(\cdot)^{(1)}$. For instance,
$V_k^{(1)}$ represents the orthnormal basis obtained from the ``previous" Arnoldi iteration.
Let $\widetilde{\bf x}_{i}=V_k^{(1)}\widetilde{\bf \psi}_i,~i=1,2,\ldots,q$.
We then construct a real matrix using
$\big\{\widetilde{\bf \psi}_{i}\}_{i=1}^q$: separate $\widetilde{\bf \psi}_{i}$ into the real
and imaginary part if it is complex, and both parts should be included and adjust $q$ if necessary. Then orthonormalize
these vectors in order to form a $k\times q$ orthonormal matrix
$W_{q}$.

We consider how to establish an Arnoldi-like relation for the new cycle, using the eigen-information retained from the ``previous" cycle. Let $\widehat{\bf r}_{\ell,k}^{(1)}(t)=V_{k+1}^{(1)}{\bf w}_{k}^{(1)}(t)\cdot\big[{\bf e}_{k}^{\rm H}{\varphi_{\ell}}(-t T_{k})\beta {\bf e}_{1}\big]$, where
\begin{equation}\label{eqn3.18}
{\bf w}_{k}^{(1)}(t)=\left[\begin{array}{c} \gamma h_{k+1,k}^{2}(I+\gamma H_{k})^{\rm -H}{\bf e}_{k}\\
-h_{k+1,k}
\end{array}\right].
\end{equation}
Thus, the residuals $\widehat{\bf r}_{\ell,k}^{(1)}(t)~(\ell \geq 0)$ are colinear with each other and are independent of $\ell$.
Denote by
$\widehat{W}_{q}=[W_{q};~{\bf 0}]\in\mathbb{R}^{(k+1)\times q}$ the matrix obtained from appending a zero row at the bottom
of $W_q$, then
\begin{eqnarray}\label{4.2}
AV_{k}^{(1)}W_{q}&\subseteq& \mbox{span}\{V_{k}^{(1)}W_{q},~\widehat{\bf r}_{\ell,k}^{(1)}(t)\}\nonumber\\
&=& \mbox{span}\Big\{V_{k+1}^{(1)}\big[\widehat{W}_{q},~{\bf w}_{k}^{(1)}(t)/\|{\bf w}_{k}^{(1)}(t)\|_2\big]\Big\}\nonumber\\
&=& \mbox{span}\{V_{k+1}^{(1)}W_{q+1}\},
\end{eqnarray}
where $W_{q+1}=\big[\widehat{W}_{q},~{\bf w}_{k}^{(1)}(t)/\|{\bf w}_{k}^{(1)}(t)\|_2\big]\in\mathbb{R}^{(k+1)\times (q+1)}$. We orthonormalize the columns of $W_{q+1}$ and still denote the resulting $(k+1)$-by-$(q+1)$ matrix by $W_{q+1}$. Let $V_{q}^{new}=V_{k}^{(1)}W_{q}$ and $V_{q+1}^{new}=V_{k+1}^{(1)}W_{q+1}$, by (\ref{4.2}),
there is a $(q+1)\times q$ matrix $\bar{H}_{q}^{new}$ such that
\begin{eqnarray*}
AV_{k}^{(1)}W_{q}=AV_{q}^{new}=(V_{k+1}^{(1)}W_{q+1})\bar{H}_{q}^{new}=V_{q+1}^{new}\bar{H}_{q}^{new},
\end{eqnarray*}
where
$\bar{H}_{q}^{new}=(V_{q+1}^{new})^{\rm H}AV_{q}^{new}$. Then we have the following relation
$$
AV_{q}^{new}=V_{q+1}^{new}\bar{H}_{q}^{new}.
$$
We then apply the standard Arnoldi process by using ${\bf v}_{q+1}^{new}=V_{q+1}^{new}(:,q+1)$ \big(i.e., the $(q+1)$-th column of $V_{q+1}^{new}$\big) as the initial vector, to form the orthonormal basis
$V_{k+1}^{(2)}$ for the new cycle
\begin{equation}
AV_{k}^{(2)}=V_{k}^{(2)}{H}_{k}^{(2)}+h_{k+1,k}^{(2)}{\bf v}_{k+1}^{(2)}{\bf e}_k^{\rm H}=V_{k+1}^{(2)}\bar{H}_k^{(2)}.
\end{equation}
Therefore, some recurrences similar to the Arnoldi
recurrence (\ref{equ22}) are generated by the thick-restarted
Arnoldi algorithm. Notice that the matrix composed of the first $(q+1)$ rows and the first $q$ columns of ${H}_{k}^{(2)}$ is full rather than upper Hessenberg.
Furthermore, one only requires to perform $(k-q)$ matrix-vector products at each cycle after the first, since the first $q$ matrix-vector products are carried out ``implicitly".

We discuss how to update the approximate solution in the new search space ${\rm span}\{V_{k}^{(2)}\}$. We first consider how to update the approximation $\widehat{\bf y}_{0,k}^{(2)}(t)$
for the matrix exponential.
To do this, we seek a vector
$\widehat{\bf z}_{0,k}^{(2)}(t)$ such that
\begin{equation}\label{equ45}
\widehat{\bf y}_{0,k}^{(2)}(t)= \widehat{\bf y}_{0,k}^{(1)}(t)+V_{k}^{(2)}\widehat{\bf z}_{0,k}^{(2)}(t)
\end{equation}
is a new approximation to ${\bf y}(t)$. We note that
\begin{eqnarray}\label{eqn3.21}
\widehat{\bf r}_{0,k}^{(2)}(t)
&=& -A \widehat{\bf y}_{0,k}^{(2)}(t)-  \widehat{\bf y}_{0,k}^{(2)}(t)^{'}\nonumber\\
&=& -A\big[V_{k}^{(1)}\widehat{\bf u}_{0,k}^{(1)}(t)+V_{k}^{(2)}\widehat{\bf z}_{0,k}^{(2)}(t)\big]-\big[V_{k}^{(1)}\widehat{\bf u}_{0,k}^{(1)}(t)^{'}+V_{k}^{(2)}  \widehat{\bf z}_{0,k}^{(2)}(t)^{'}\big]\nonumber\\
&=& \big[-AV_{k}^{(1)}\widehat{\bf u}_{0,k}^{(1)}(t)- V_{k}^{(1)}\widehat{\bf u}_{0,k}^{(1)}(t)^{'}\big]-AV_{k}^{(2)}\widehat{\bf z}_{0,k}^{(2)}(t)-V_{k}^{(2)}  \widehat{\bf z}_{0,k}^{(2)}(t)^{'}\nonumber\\
&=& \widehat{\bf r}_{0,k}^{(1)}(t)-AV_{k}^{(2)}\widehat{\bf z}_{0,k}^{(2)}(t)-V_{k}^{(2)}\widehat{\bf z}_{0,k}^{(2)}(t)^{'}.
\end{eqnarray}
Recall from the thick-restarting procedure that $ \widehat{\bf r}_{0,k}^{(1)}(t) \in {\rm span}\{{V_{k+1}^{(2)}}\}$, so there exists a vector ${\bf c}^{(2)}_{0,k}(t)$ such that $\widehat{\bf r}_{0,k}^{(1)}(t)= V_{k+1}^{(2)}{\bf c}_{0,k}^{(2)}(t)$. Let
$$
\widehat{\bf r}_{0,k}^{(2)}(t) \bot {\rm span}\{(I+\gamma A)V_{k}^{(2)}\},
$$
that is,
$$
\big[(I+\gamma A)V_{k}^{(2)}\big]^{\rm H}\big[V_{k+1}^{(2)}{\bf c}_{0,k}^{(2)}(t)-AV_{k}^{(2)}\widehat{\bf z}_{0,k}^{(2)}(t)-V_{k}^{(2)}\widehat{\bf z}_{0,k}^{(2)}(t)^{'}\big]
={\bf 0},
$$
we obtain
\begin{equation}\label{equ47}
\left\{\begin{array}{l}  \widehat{\bf z}_{0,k}^{(2)}(t)^{'}=-\Xi_{k}^{(2)}\bar{H}_{k}^{(2)}\widehat{\bf z}_{0,k}^{(2)}(t)+\Xi_{k}^{(2)}{\bf c}_{0,k}^{(2)}(t),\\
\widehat{\bf z}_{0,k}^{(2)}(0)={\bf 0},
\end{array}
\right.
\end{equation}
where
\begin{equation}\label{eqnn3.24}
\Xi_{k}^{(2)}=(I+\gamma H_k^{(2)})^{\rm -H}(\bar{I}+\gamma \bar{H}_k^{(2)})^{\rm H},
\end{equation}
and $\bar{I}$ is the $(k+1)\times k$ matrix being the $k\times k$ identity matrix with an additional zero row at the bottom.

Thus, we update the approximate solution to ${\bf y}(t)$ via solving a small-sized ODE (\ref{equ47}). The residual is
\begin{eqnarray}\label{eqn3.25}
\widehat{\bf r}_{0,k}^{(2)}(t)
&=& {V}_{k+1}^{(2)} {\bf c}_{0,k}^{(2)}(t)-{V}_{k+1}^{(2)}\bar{H}_{k}^{(2)}\widehat{\bf z}_{0,k}^{(2)}(t)-{V}_{k}^{(2)}\widehat{\bf z}_{0,k}^{(2)}(t)^{'} \nonumber\\
&=& {V}_{k+1}^{(2)}\Bigg[{\bf c}_{0,k}^{(2)}(t)-\bar{H}_{k}^{(2)}\widehat{\bf z}_{0,k}^{(2)}(t)-\left[\begin{array}{c}   \widehat{\bf z}_{0,k}^{(2)}(t)^{'}\\
0
\end{array}\right]\Bigg],
\end{eqnarray}
with
\begin{equation}\label{eqn3.24}
\| \widehat{\bf r}_{0,k}^{(2)}(t) \| _{2}=\Bigg\|{\bf c}_{0,k}^{(2)}(t)-\bar{H}_{k}^{(2)}\widehat{\bf z}_{0,k}^{(2)}(t)-\left[\begin{array}{c}   \widehat{\bf z}_{0,k}^{(2)}(t)^{'}\\
0
\end{array}\right]\Bigg\|_{2}.
\end{equation}

Next we discuss how to update the approximate solution for the $\varphi_{\ell}~(\ell \geq 1)$ functions during cycles. Similarly, given the new search subspace $V_k^{(2)}$ and the approximation $\widehat{\bf y}_{\ell,k}^{(1)}(t)$ obtained from the ``previous" cycle, we seek a vector $\widehat{\bf z}_{\ell,k}^{(2)}(t)$ such that
\begin{equation}\label{eqn3.29}
\widehat{\bf y}_{\ell,k}^{(2)}(t)=\widehat{\bf y}_{\ell,k}^{(1)}(t)+V_{k}^{(2)}\widehat{\bf z}_{\ell,k}^{(2)}(t)
\end{equation}
is the new approximation to ${\bf y}(t)$. The residual is
\begin{eqnarray}
\widehat{\bf r}_{\ell,k}^{(2)}(t)
&=& -A\widehat{\bf y}_{\ell,k}^{(2)}(t)-\frac{\ell}{t}\widehat{\bf y}_{\ell,k}^{(2)}(t)+\frac{1}{t(\ell-1)!}{\bf v}-\widehat{\bf y}_{\ell,k}^{(2)}(t)'\nonumber\\
&=& -A\bigg[\widehat{\bf y}_{\ell,k}^{(1)}(t)+V_{k}^{(2)}\widehat{\bf z}_{\ell,k}^{(2)}(t)\bigg]-\frac{\ell}{t}\bigg[\widehat{\bf y}_{\ell,k}^{(1)}(t)+V_{k}^{(2)}\widehat{\bf z}_{\ell,k}^{(2)}(t)\bigg]+\frac{1}{t(\ell-1)!}{\bf v}-\bigg[\widehat{\bf y}_{\ell,k}^{(1)}(t)'+V_{k}^{(2)}\widehat{\bf z}_{\ell,k}^{(2)}(t)'\bigg]\nonumber\\
&=& \Big[-A\widehat{\bf y}_{\ell,k}^{(1)}(t)-\frac{\ell}{t}\widehat{\bf y}_{\ell,k}^{(1)}(t)+\frac{1}{t(\ell-1)!}{\bf v}-\widehat{\bf y}_{\ell,k}^{(1)}(t)'\Big]\nonumber\\
&& -AV_{k}^{(2)}\widehat{\bf z}_{\ell,k}^{(2)}(t)-\frac{\ell}{t}V_{k}^{(2)}\widehat{\bf z}_{\ell,k}^{(2)}(t)-V_{k}^{(2)}\widehat{\bf z}_{\ell,k}^{(2)}(t)'.
\end{eqnarray}
It follows from the thick-restarting strategy that
\begin{eqnarray*}
\widehat{\bf r}_{\ell,k}^{(1)}(t)
&=& -A\widehat{\bf y}_{\ell,k}^{(1)}(t)-\frac{\ell}{t}\widehat{\bf y}_{\ell,k}^{(1)}(t)+\frac{1}{t(\ell-1)!}{\bf v}-\widehat{\bf y}_{\ell,k}^{(1)}(t)' \in {\rm span}\big\{V_{k+1}^{(2)}\big\}.
\end{eqnarray*}
So there exists a vector ${\bf c}_{\ell,k}^{(2)}(t)$ such that $\widehat{\bf r}_{\ell,k}^{(1)}(t)=V_{k+1}^{(2)}{\bf c}_{\ell,k}^{(2)}(t)$. Let
\begin{equation*}
\widehat{\bf r}_{\ell,k}^{(2)}(t)=V_{k+1}^{(2)}{\bf c}_{\ell,k}^{(2)}(t)-AV_{k}^{(2)}\widehat{\bf z}_{\ell,k}^{(2)}(t)-\frac{\ell}{t}V_{k}^{(2)}\widehat{\bf z}_{\ell,k}^{(2)}(t)-V_{k}^{(2)}\widehat{\bf z}_{\ell,k}^{(2)}(t)' \perp {\rm span}\{(I+\gamma A)V_{k}^{(2)}\},
\end{equation*}
i.e.,
$$
\big[(I+\gamma A)V_{k}^{(2)}\big]^{\rm H}\Big[V_{k+1}^{(2)}{\bf c}_{\ell,k}^{(2)}(t)-AV_{k}^{(2)}\widehat{\bf z}_{\ell,k}^{(2)}(t)-\frac{\ell}{t}V_{k}^{(2)}\widehat{\bf z}_{\ell,k}^{(2)}(t)-V_{k}^{(2)}\widehat{\bf z}_{\ell,k}^{(2)}(t)'\Big]={\bf 0},
$$
then
\begin{equation}\label{eqn3.41}
\left\{\begin{array}{l}
\widehat{\bf z}_{\ell,k}^{(2)}(t)^{'}=-\big[\Xi_{k}^{(2)}\bar{H}_{k}^{(2)}+\frac{\ell}{t}I\big]\widehat{\bf z}_{\ell,k}^{(2)}(t)+\Xi_{k}^{(2)}{\bf c}_{\ell,k}^{(2)}(t),\\
\widehat{\bf z}_{\ell,k}^{(2)}(0)={\bf 0},
\end{array}
\right.
\end{equation}
where $\Xi_{k}^{(2)}$ is defined in (\ref{eqnn3.24}). Therefore, we update (\ref{eqn3.29}) by solving the above small-sized ODE.
The residual is
\begin{eqnarray}\label{eqn3.42}
\widehat{\bf r}_{\ell,k}^{(2)}(t)
&=& {V}_{k+1}^{(2)} {\bf c}_{\ell,k}^{(2)}(t)-{V}_{k+1}^{(2)}\bar{H}_{k}^{(2)}\widehat{\bf z}_{\ell,k}^{(2)}(t)-\frac{\ell}{t}V_k^{(2)}\widehat{\bf z}_{\ell,k}^{(2)}(t)-{V}_{k}^{(2)}\widehat{\bf z}_{\ell,k}^{(2)}(t)^{'} \nonumber\\
&=& {V}_{k+1}^{(2)}\Bigg[{\bf c}_{\ell,k}^{(2)}(t)-\big(\bar{H}_{k}^{(2)}+\frac{\ell}{t}\bar{I}\big)\widehat{\bf z}_{\ell,k}^{(2)}(t)-\left[\begin{array}{c}   \widehat{\bf z}_{\ell,k}^{(2)}(t)^{'}\\
0
\end{array}\right]\Bigg],\quad \ell\geq 1,
\end{eqnarray}
and
\begin{equation}\label{eqn3.43}
\| \widehat{\bf r}_{\ell,k}^{(2)}(t) \| _{2}=\Bigg\|{\bf c}_{\ell,k}^{(2)}(t)-\big[\bar{H}_{k}^{(2)}+\frac{\ell}{t}\bar{I}\big]\widehat{\bf z}_{\ell,k}^{(2)}(t)-\left[\begin{array}{c}
\widehat{\bf z}_{\ell,k}^{(2)}(t)^{'}\\
0
\end{array}\right]\Bigg\|_{2},\quad \ell\geq 1.
\end{equation}
When $\ell=0$, it is seen that (\ref{eqn3.41}) and (\ref{eqn3.42}) reduce to (\ref{equ47}) and (\ref{eqn3.25}), respectively.

In summary, we propose the main algorithm of this paper for solving (\ref{eqn1.3}).
\begin{algorithm}
{\bf  A thick-restarted harmonic Arnoldi algorithm for the action of $\varphi$-functions on a vector~(TRHA)}\\
%---------------------------------------------------------------------------------------------------------------------------\\
 {{\bf Step 1.} Given the matrix $A$, the vector ${\bf v}$, the values of $t$ and $s$, a shift $\gamma$, as well as a convergence tolerance {\rm tol}. Specify $k$, the steps of the Arnoldi process, and $q$, the number of approximate eigenvectors which are retained from one cycle to the next; \\
 {\bf Step 2.} Run the $k$-step Arnoldi process to form $V_{k+1}$ and ${H}_{k}$. Compute the approximate solutions $\widehat{\bf y}_{\ell,k}(t)~(\ell=0,1,\ldots,s)$. If all the residual norms are below {\rm tol} {\rm(}see {\rm(}\ref{eqn3.5}{\rm)} and {\rm(}\ref{eqnn3.34}{\rm))}, then stop, else compute the primitive harmonic Ritz pairs $(\widetilde{\lambda}_i,\widetilde{\psi}_i),~i=1,2,\ldots,k$, and select $q$ of them. Go to {\bf Step 4};\\
 {\bf Step 3.} Run the remaining $(k-q)$ steps of Arnoldi process to form $V_{k+1}$ and ${H}_{k}$, using the last column of $V_{q+1}$ as the initial vector.
 Update the approximation $\widehat{\bf y}_{\ell,k}(t)$ by solving {\rm(}\ref{eqn3.41}{\rm)} $(\ell=0,1,\ldots,s)$, if all the residual norms are below {\rm tol}
  {\rm(}see  {\rm(}\ref{eqn3.43}{\rm)}{\rm)},
 then stop, otherwise compute primitive Ritz pairs $(\widetilde{\lambda}_i,\widetilde{\psi}_i),~i=1,2,\ldots,k$, and select $q$ of them; \\
 {\bf Step 4.} Orthonormalize the $\{\widetilde{\bf \psi}_{i}\}'s,~i=1,2,\ldots,q$, to form a real $k$-by-$q$ matrix $W_{q}=[\widetilde{\bf \psi}_{1},\widetilde{\bf \psi}_{2},\ldots,\widetilde{\bf \psi}_{q}]$. If $\widetilde{\bf \psi}_{i}$ is
 complex, separate it into the real part and the imaginary part, both parts should be included, and adjust $q$ if necessary {\rm(}increasing or decreasing $q$ by 1{\rm)};\\
 {\bf Step 5.} Extend $W_{q}$ to a $(k+1) \times q$ matrix $\widehat{W}_{q}=[W_{q };~{\bf 0}]$, where {\bf 0} is a zero row vector. Let
 $W_{q+1}=[\widehat{W}_{p},~{\bf w}_{k}(t)/\|{\bf w}_{k}(t)\|_2]$. Then orthonormalize the columns of $W_{q+1}$ to yield an orthonormal matrix with size $(k+1)\times(q+1);$\\
 {\bf Step 6.} Form the portions of new $\bar{H}_{k}$ and $V_{k+1}$ by using the old $\bar{H}_{k}$ and $V_{k+1}$: Let $\bar{H}_{q}^{new}=W_{q+1}^{\rm H}\bar{H}_{k}W_{q}$ and
 $V_{q+1}^{new}=V_{k+1}W_{q+1}$, then set $\bar{H}_{q}=\bar{H}_{q}^{new}$ and $V_{q+1}=V_{q+1}^{new}$. Go to {\bf Step 3}.
}
\end{algorithm}

\begin{rem}
Two remarks are in order. First, since the residuals $\widehat{\bf r}_{\ell,k}(t)~(\ell=0,1,2,\ldots,s)$ are colinear with each other, one can solve the vectors
$\varphi_{\ell}(-tA){\bf v}~(\ell=0,1,\ldots,s)$
simultaneously, and compute them in the same search subspace.
Second, as a by-product, we can also present a thick-restarted Arnoldi algorithm for the $\varphi$-functions. The difference is that one evaluates the Arnoldi approximations ${\bf y}_{\ell,k}(t)$ via the orthogonal projection technique {\rm(}\ref{eqn2.5}{\rm)} and {\rm(}\ref{eqn2.18}{\rm)}, and augments the search subspace with the Ritz vectors rather than harmonic Ritz vectors.
\end{rem}

\section{Relationship between the error and the residual of the approximations}

\setcounter{equation}{0}

In this section, we investigate the relationship between the error and the residual of the (harmonic) Arnoldi approximation for $\varphi$-functions.
Let $\Gamma$ be a closed contour that encloses
the spectra of $-tA$ and $-tH_k$. Let ${\bf y}(t)=\varphi_{\ell}(-tA){\bf v}$ and let ${\bf y}_{\ell,k}(t)=V_{k}\varphi_{\ell}(-t H_{k})\beta{\bf e}_{1}$ be the approximation from the Arnoldi method, where ${\bf y}_{0,k}(t)\equiv{\bf y}_k(t)$ is the Arnoldi approximation for matrix exponential. If $\varphi_{\ell}~(\ell\geq 0)$ are analytic on and inside the closed contour $\Gamma$, from the Dunford-Taylor integral representation \cite{Higham}, we obtain
\begin{equation}\label{equ61}
{\bf y}(t)=\varphi_{\ell}(-tA){\bf v}=\frac{1}{2\pi \bf i}{\int_\Gamma \varphi_{\ell}(z)(zI+tA)^{-1}}{\bf v}dz,\quad \ell=0,1,2,\ldots
\end{equation}
where ${\bf i}^{2}=-1$. Moreover,
\begin{equation}\label{equ62}
{\bf y}_{\ell,k}(t)=V_{k}\varphi_{\ell}(-tH_{k})\beta{\bf e}_{1}=\frac{1}{2\pi \bf i}{\int_\Gamma \varphi_{\ell}(z)V_{k}(zI+t H_{k})^{-1}}\beta{\bf e}_{1}dz,\quad \ell=0,1,2,\ldots
\end{equation}

We have the following theorem on the Arnoldi approximation for $\varphi$-functions.
\begin{theorem}
 Denote by ${\bf e}_{\ell,k}(t)={\bf y}(t)-{\bf y}_{\ell,k}(t)$ the error, and by ${\bf r}_{\ell,k}(t)$ the residual with respect to ${\bf y}_{\ell,k}(t)$, where ${\bf r}_{0,k}(t)\equiv{\bf r}_k(t)$ is the residual of the Arnoldi approximation to matrix exponential.
Assume that ${\bf e}_{k}^{\rm H}\varphi_{\ell}(-t H_{k})\beta{\bf e}_{1}\neq 0$, and denote
$$
f_{H_{k}}(z)=t\varphi_{\ell}(z)\cdot\frac{{\bf e}_{k}^{\rm H}(zI+t H_{k})^{-1}\beta{\bf e}_{1}}{{\bf e}_{k}^{\rm H}\varphi_{\ell}(-t H_{k})\beta{\bf e}_{1}}.
$$
Then we have
$$
{\bf e}_{\ell,k}(t)=\frac{1}{2\pi \bf i}{\int_\Gamma f_{H_{k}}(z)(zI+tA)^{-1}}\cdot{\bf r}_{\ell,k}(t)dz,
$$
and
\begin{equation}
\|{\bf e}_{\ell,k}(t)\|_2 \leq \Big\| \frac{1}{2\pi \bf i}{\int_\Gamma f_{H_{k}}(z)(zI+tA)^{-1}}dz\Big\|_2\cdot\|{\bf r}_{\ell,k}(t)\|_2.
\end{equation}
\end{theorem}
\begin{proof}
It follows that
\begin{eqnarray*}
{\bf e}_{\ell,k}(t)
&=& {\bf y}(t)-{\bf y}_{\ell,k}(t)\nonumber \\
&=& \frac{1}{2\pi \bf i}{\int_\Gamma \varphi_{\ell}(z)\big[(zI+tA)^{-1}{\bf v}-V_{k}(zI+t H_{k})^{-1}}\beta{\bf e}_{1}\big]dz,
\end{eqnarray*}
and
$$
(zI+tA)^{-1}{\bf v}-V_{k}(zI+t H_{k})^{-1}\beta{\bf e}_{1}=(zI+tA)^{-1}\big[{\bf v}-(zI+tA)V_{k}(zI+t H_{k})^{-1}\beta{\bf e}_{1}\big].
$$
From the Arnoldi relation (\ref{equ22}), we get
\begin{equation*}
{\bf v}-(zI+tA)V_{k}(zI+t H_{k})^{-1}\beta{\bf e}_{1}=-th_{k+1,k}\big[{\bf e}_{k}^{\rm H}(zI+t H_{k})^{-1}\beta{\bf e}_{1}\big]{\bf v}_{k+1}.
\end{equation*}
Recall that ${\bf r}_{\ell,k}(t)=-h_{k+1,k}\big[{\bf e}_{k}^{\rm H}\varphi_{\ell}(-t H_{k})\beta{\bf e}_{1}\big]{\bf v}_{k+1}$, thus
\begin{eqnarray}\label{equ66-2}
{\bf e}_{\ell,k}(t)
&=& \frac{1}{2\pi \bf i}{\int_\Gamma \varphi_{\ell}(z)}(zI+tA)^{-1} \Big[-t h_{k+1,k}\big[{\bf e}_{k}^{\rm H}(zI+t H_{k})^{-1}\beta{\bf e}_{1}\big]{\bf v}_{k+1}\Big]dz \\
&=& \frac{1}{2\pi \bf i}{\int_\Gamma \varphi_{\ell}(z)(zI+tA)^{-1}\frac{{\bf e}_{k}^{\rm H}(zI+t H_{k})^{-1}\beta{\bf e}_{1}}{{\bf e}_{k}^{\rm H}\varphi_{\ell}(-t H_{k})\beta{\bf e}_{1}}}\Big[-t h_{k+1,k}\big[{\bf e}_{k}^{\rm H}\varphi_{\ell}(-t H_{k})\beta{\bf e}_{1}\big]{\bf v}_{k+1}\Big]dz \nonumber \\
&=&\frac{1}{2\pi \bf i}{\int_\Gamma t\varphi_{\ell}(z)\frac{{\bf e}_{k}^{\rm H}(zI+t H_{k})^{-1}\beta{\bf e}_{1}}{{\bf e}_{k}^{\rm H}\varphi_{\ell}(-t H_{k})\beta{\bf e}_{1}}}\cdot(zI+tA)^{-1}{\bf r}_{\ell,k}(t)dz.\nonumber
\end{eqnarray}
Let $f_{H_{k}}(z)=t\varphi_{\ell}(z)\frac{{\bf e}_{k}^{\rm H}(zI+t H_{k})^{-1}\beta{\bf e}_{1}}{{\bf e}_{k}^{\rm H}\varphi_{\ell}(-t H_{k})\beta{\bf e}_{1}}$, and notice that
\begin{eqnarray}
\frac{1}{2\pi \bf i}{\int_\Gamma t\varphi_{\ell}(z)}\frac{{\bf e}_{k}^{\rm H}(zI+t H_{k})^{-1}\beta{\bf e}_{1}}{{\bf e}_{k}^{\rm H}\varphi_{\ell}(-t H_{k})\beta{\bf e}_{1}}dz&=&\frac{1}{2\pi \bf i}{\int_\Gamma f_{H_{k}}(z)dz}=t.
\end{eqnarray}
Therefore, we have from (\ref{equ66-2}) that
\begin{eqnarray*}\label{equ63}
\|{\bf e}_{\ell,k}(t)\|_2 &=&\Big\| \frac{1}{2\pi \bf i}{\int_\Gamma f_{H_{k}}(z)(zI+tA)^{-1}}dz\cdot{\bf r}_{\ell,k}(t)\Big\|_2 \nonumber \\
&\leq& \Big\| \frac{1}{2\pi \bf i}{\int_\Gamma f_{H_{k}}(z)(zI+tA)^{-1}}dz\Big\|_2\cdot\|{\bf r}_{\ell,k}(t)\|_2. \nonumber
\end{eqnarray*}
\end{proof}

Now we focus on error estimates of a class of special matrices.
Denote the numerical range of $A$ by $W(A)$, i.e.,
$$W(A)=\Big\{\frac{(A{\bf x},{\bf x})}{({\bf x},{\bf x})},{\bf 0}\neq{\bf x}\in \mathbb{C}^{n}\Big\},$$
where $(\cdot,\cdot)$ represents the Euclidean inner product. Note that $W(H_{k})\subseteq W(A)$ holds for each $k$.
For $a\geq 0$ and $0\leq \vartheta \leq \pi/2$, we define the set
$$
\Sigma_{\vartheta,a}=\big\{z\in \mathbb{C}:\big|{\rm arg}(z-a)\big|\leq \vartheta \big\},
$$
where ${\rm arg}(\cdot)$ denotes the argument of a complex number; and assume that
\begin{eqnarray}\label{equ64}
W(A)\subset \Sigma_{\vartheta,a}.
\end{eqnarray}
It is known that this assumption holds in important applications such as parabolic partial differential equations.

Similar to \cite{FIS-2008}, our analysis is based on the integral representation \cite{FIS-2008,EN2000}
$$
{\bf y}(t)=\frac{1}{t^{\ell}}\lim_{N\rightarrow \infty}\frac{1}{2\pi {\bf i}}\int_{\varepsilon-{\bf i}N}^{\varepsilon+{\bf i}N}\frac{\exp(tz)}{z^{\ell}}(z I+A)^{-1}{\bf v}dz,\quad\ell=0,1,2,\ldots,
$$
which, under our assumptions, holds for every $\varepsilon>0$ with uniform convergence when $t$ is chosen in compact intervals of $(0,+\infty)$.
Given $\varepsilon>0$, if we set $z=\varepsilon +{\bf i}\rho$, then \cite{FIS-2008}
$${\bf y}(t)=\frac{\exp(t\varepsilon)}{t^{\ell}}\lim_{N\rightarrow \infty}\frac{1}{2\pi}\int_{-N}^{+N}\frac{\exp({\bf i}t\rho)}{(\varepsilon +{\bf i}\rho)^{\ell}}\big((\varepsilon +{\bf i}\rho) I+A\big)^{-1}{\bf v}d\rho.$$

Suppose that $A\in\mathbb{R}^{n\times n}$ is a real matrix. Consider $\mu_{j}=a_{j}+{\bf i}b_{j}$ ($j=1,2,\ldots,k$) the eigenvalues of matrix $H_{k}$ arranging as $\mu_{1},\ldots,\mu_{k_{1}}$ the real ones and $\mu_{k_{1}+1},\ldots,\mu_{k}$ the
complex conjugate ones. Let
$$
r_{j}=\big((\varepsilon+a_{j})^{2}+b_{j}^{2}\big)^{1/2},\quad R=\max_{1\leq j \leq k}r_{j},
$$
and
$$\omega_{k}(\varepsilon)=\prod_{j=1}^{k}\big(r_{j}(\varepsilon+a_{j})\big)^{1/2}.$$
Define
$$\varsigma_{k}(\rho)=\prod_{j=1}^{k_{1}}(1+\rho^{2}/r_{j}^{2})^{1/2}\prod_{j=k_{1}+1}^{k}(1+\rho^{2}/r_{j}^{2})^{1/4},$$
and
\begin{equation}\label{eqn4.7}
d_{k}(\rho)=\frac{\prod_{j=1}^{k-1}h_{j+1,j}}{\omega_{k}(\varepsilon)\varsigma_{k}(\rho)}.
\end{equation}
Under the above assumptions, we can prove the following result whose proof is along the line of Proposition 5 of \cite{FIS-2008}.
\begin{theorem}\label{theo4.2}
Let $\varepsilon>0$ and suppose $k+k_{1}+2\ell\geq 4$. Then for the Arnoldi approximation, we have
\begin{equation}\label{equ64-1}
\|{\bf e}_{\ell,k}(t)\|_{2}\leq c_{\ell,k}\cdot\int_{0}^{\infty}\frac{(1+\rho^{2}/\varepsilon^{2})^{-\ell/2}}{\varsigma_{k}(\rho)}d\rho\cdot\|{\bf r}_{\ell,k}(t)\|_{2}
\end{equation}
where $c_{\ell,k}=\frac{\exp(t\varepsilon)}{\pi (t\varepsilon)^{\ell}\omega_{k}(\varepsilon)(\varepsilon+a)}\cdot\frac{\prod_{j=1}^{k-1}h_{j+1,j}}{|{\bf e}_{k}^{\rm H}\varphi_{\ell}(-t H_{k}){\bf e}_{1}|}$.
\end{theorem}

\begin{proof}
It follows that
\begin{eqnarray*}
{\bf e}_{\ell,k}(t)
&=& {\bf y}(t)-{\bf y}_{\ell,k}(t)\nonumber \\
&=& \frac{\exp(t\varepsilon)}{t^{\ell}}\lim_{N\rightarrow \infty}\frac{1}{2\pi}{\int_{-N}^{+N} \frac{\exp({\bf i}t\rho)}{(\varepsilon +{\bf i}\rho)^{\ell}}\Big[\big((\varepsilon +{\bf i}\rho)I+A\big)^{-1}{\bf v}-V_{k}\big((\varepsilon +{\bf i}\rho)I+H_{k}\big)^{-1}}\beta{\bf e}_{1}\Big]d\rho,
\end{eqnarray*}
and
$$\big((\varepsilon +{\bf i}\rho)I+A\big)^{-1}{\bf v}-V_{k}\big((\varepsilon +{\bf i}\rho)I+H_{k}\big)^{-1}\beta{\bf e}_{1}$$
$$=\big((\varepsilon +{\bf i}\rho)I+A\big)^{-1}\Big[{\bf v}-\big((\varepsilon +{\bf i}\rho)I+A\big)V_{k}\big((\varepsilon +{\bf i}\rho)I+H_{k}\big)^{-1}\beta{\bf e}_{1}\Big].$$
From the Arnoldi relation (\ref{equ22}), we obtain
\begin{equation*}
{\bf v}-\big((\varepsilon +{\bf i}\rho)I+A\big)V_{k}\big((\varepsilon +{\bf i}\rho)I+H_{k}\big)^{-1}\beta{\bf e}_{1}=-h_{k+1,k}\Big[{\bf e}_{k}^{\rm H}\big((\varepsilon +{\bf i}\rho)I+H_{k}\big)^{-1}\beta{\bf e}_{1}\Big]{\bf v}_{k+1}.
\end{equation*}
Recall that ${\bf r}_{\ell,k}(t)=-h_{k+1,k}\big[{\bf e}_{k}^{\rm H}\varphi_{\ell}(-t H_{k})\beta{\bf e}_{1}\big]{\bf v}_{k+1}$, thus
\begin{equation}\label{equ65}
{\bf e}_{\ell,k}(t)=\frac{\exp(t\varepsilon)}{t^{\ell}}\lim_{N\rightarrow \infty}\frac{1}{2\pi}\int_{-N}^{+N}\frac{\exp({\bf i}t\rho)}{(\varepsilon +{\bf i}\rho)^{\ell}}\big((\varepsilon +{\bf i}\rho)I+A\big)^{-1}\frac{{\bf e}_{k}^{\rm H}\big((\varepsilon +{\bf i}\rho)I+H_{k}\big)^{-1}{\bf e}_{1}}{{\bf e}_{k}^{\rm H}\varphi_{\ell}(-t H_{k}){\bf e}_{1}}{\bf r}_{\ell,k}(t)d\rho.
\end{equation}

It follows from Lemma 2 of \cite{FIS-2008} that
\begin{equation}\label{eqnn49}
\big|{\bf e}_{k}^{\rm H}\big((\varepsilon +{\bf i}\rho)I+H_{k}\big)^{-1}{\bf e}_{1}\big|=\Big|{\rm det}\big((\varepsilon +{\bf i}\rho)I+H_{k}\big)^{-1}\prod_{j=1}^{k-1}h_{j+1,j}\Big|.
\end{equation}
Moreover, we have from Proposition 5 of \cite{FIS-2008} that
%\begin{equation}\label{eqn4.9}
%\big|{\rm det}\big((\varepsilon +{\bf i}\rho)I+H_{k}\big)\big|\geq \omega_{k}(\varepsilon)\varsigma_{k}(\varepsilon),
%\end{equation}
\begin{equation}\label{equ66}
\big|{\bf e}_{k}^{\rm H}\big((\varepsilon +{\bf i}\rho)I+H_{k}\big)^{-1}{\bf e}_{1}\big|\leq d_{k}(\rho),
\end{equation}
and
\begin{equation}\label{equ67}
\|\big((\varepsilon +{\bf i}\rho)I+A\big)^{-1}\|_{2}\leq (\varepsilon+a)^{-1}.
\end{equation}
%which arises from the inequality $\|(z I+A)^{-1}\|_{2}\leq dist(-z,W(A))^{-1}$ and by assumption (\ref{equ64}).
By means of (\ref{eqnn49})--(\ref{equ67}), we obtain from (\ref{eqn4.7}) and (\ref{equ65}) that
$$\|{\bf e}_{\ell,k}(t)\|_{2}\leq \frac{\exp(t\varepsilon)}{\pi t^{\ell}
(\varepsilon+a)}\cdot\frac{1}{|{\bf e}_{k}^{\rm H}\varphi_{\ell}(-t H_{k}){\bf e}_{1}|}\cdot\int_{0}^{\infty}\frac{d_{k}(\rho)}{(\varepsilon^{2}+\rho^{2})^{\ell/2}}d\rho\cdot\|{\bf r}_{\ell,k}(t)\|_{2}$$
$$=c_{\ell,k}\cdot\int_{0}^{\infty}\frac{(1+\rho^{2}/\varepsilon^{2})^{-\ell/2}}{\varsigma_{k}(\rho)}d\rho\cdot\|{\bf r}_{\ell,k}(t)\|_{2}.$$
Thus (\ref{equ64-1}) is proved. We notice that if $k+k_{1}+2\ell\geq 4$, then the integral in (\ref{equ64-1}) converges.
\end{proof}

The following result avoids the use of a quadrature rule for evaluating the integral in (\ref{equ64-1}). The proof is similar to Proposition 6 of \cite{FIS-2008} and is omitted.
\begin{theorem}\label{theo4.3}
Let $\varepsilon>0$ and suppose that $k+k_{1}\geq 4$. Under the above notations, we have
$$\|{\bf e}_{\ell,k}(t)\|_{2}\leq c_{\ell,k}\cdot C_{k}\cdot\|{\bf r}_{\ell,k}(t)\|_{2},$$
where
$$C_{k}=\frac{\sqrt{\pi}}{2\sqrt{S_{1}}}+\exp\big(-\varepsilon^{2}S_{2}\big)
(R-\varepsilon)+\Big(\frac{\varepsilon}{\sqrt{\varepsilon^{2}+R^{2}}}\Big)^{\ell}\frac{\pi R}{2^{(\frac{k+k_{1}}{4}+1)}},$$
with
$$S_{1}=\frac{\ell}{4\varepsilon^{2}}+\frac{1}{2}\sum_{j=1}^{k_{1}}\frac{1}{r_{j}^{2}+
\varepsilon^{2}}+\frac{1}{4}\sum_{j=k_{1}+1}^{k}\frac{1}{r_{j}^{2}+\varepsilon^{2}},$$
and
$$S_{2}=\frac{\ell}{2(\varepsilon^{2}+R^{2})}+\frac{1}{2}\sum_{j=1}^{k_{1}}\frac{1}{r_{j}^{2}+R^{2}}+\frac{1}{4}\sum_{j=k_{1}+1}^{k}\frac{1}{r_{j}^{2}+R^{2}}.$$
\end{theorem}

%\begin{proof}
%Let
%$$\Psi_{k}(\rho)=\frac{(1+\rho^{2}/\varepsilon^{2})^{-\ell/2}}{\varsigma_{k}(\rho)},$$
%we approximate the integral in (\ref{equ64-1}), exploiting the relationship \cite{FIS-2008}
%\begin{equation}\label{equ68}
%\Psi_{k}(\rho)\leq \exp\big(-\rho^{2}(\frac{\ell}{2(\varepsilon^{2}+\rho^{2})}+\frac{1}{2}\sum_{j=1}^{k_{1}}\frac{1}{r_{j}^{2}+\rho^{2}}+\frac{1}{4}\sum_{j=k_{1}+1}^{k}\frac{1}{r_{j}^{2}+\rho^{2}})\big),
%\end{equation}
%which comes form $1-x\leq \exp(-x),~0\leq x\leq 1$.
%
%At first, we have from (\ref{equ68})
%$$\int_{0}^{\varepsilon}\Psi_{k}(\rho)d\rho\leq \int_{0}^{\varepsilon}\exp(-\rho^{2}S_{k}^{(1)})d\rho$$
%$$=\frac{1}{\sqrt{S_{k}^{(1)}}}\int_{0}^{\varepsilon\sqrt{S_{k}^{(1)}}}\exp(-x^{2})dx\leq \frac{\sqrt{\pi}}{2\sqrt{S_{k}^{(1)}}}.$$
%Furthermore (\ref{equ68}) yields
%$$\int_{\varepsilon}^{R}\Psi_{k}(\rho)d\rho\leq \int_{\varepsilon}^{R}\exp(-\rho^{2}S_{k}^{(2)})d\rho\leq \exp(-\varepsilon^{2}S_{k}^{(2)})(R-\varepsilon).$$
%Finally it can be verified that, for $k+k_{1}\geq 4$, it holds
%$$\int_{R}^{+\infty}\Psi_{k}(\rho)d\rho\leq (\frac{\varepsilon}{\sqrt{\varepsilon^{2}+R^{2}}})^{\ell}\int_{R}^{+\infty}\frac{1}{(1+\rho^{2}/R^{2})^{\frac{k+k_{1}}{4}}}d\rho  $$
%$$\leq(\frac{\varepsilon}{\sqrt{\varepsilon^{2}+R^{2}}})^{\ell}\frac{\pi R}{2^{(\frac{k+k_{1}}{4}+1)}}.$$
%This completes the proof.
%
%\end{proof}

Next, we consider the harmonic Arnoldi approximation.
Let $\widehat{\Gamma}$ be a closed contour that encloses
the spectra of $-tA$ and $-tT_k$. Let $\widehat{\bf y}_{\ell,k}(t)=V_{k}\varphi_{\ell}(-t T_{k})\beta{\bf e}_{1}$ be the approximation from the harmonic Arnoldi method, where $\widehat{\bf y}_{0,k}(t)\equiv\widehat{\bf y}_k(t)$ is the harmonic Arnoldi approximation for matrix exponential. If $\varphi_{\ell}~(\ell \geq 0)$ are analytic on and inside the closed contour $\widehat{\Gamma}$, we obtain from the Dunford-Taylor integral representation that
\begin{equation}
\widehat{\bf y}_{\ell,k}(t)=V_{k}\varphi_{\ell}(-t T_{k})\beta{\bf e}_{1}=\frac{1}{2\pi \bf i}{\int_{\widehat{\Gamma}} \varphi_{\ell}(z)V_{k}(zI+t T_{k})^{-1}}\beta{\bf e}_{1}dz,\quad \ell=0,1,2,\ldots
\end{equation}
We are ready to present the following theorem on the relation between the error and the residual of the harmonic Arnoldi approximation $\widehat{\bf y}_{\ell,k}(t)$.
\begin{theorem}
 Denote by $\widehat{\bf e}_{\ell,k}(t)={\bf y}(t)-\widehat{\bf y}_{\ell,k}(t)$ the error, and by $\widehat{\bf r}_{\ell,k}(t)$ the residual with respect to $\widehat{\bf y}_{\ell,k}(t)$, where $\widehat{\bf r}_{0,k}(t)\equiv\widehat{\bf r}_k(t)$ is the residual of the harmonic Arnoldi approximation to matrix exponential.
Assume that ${\bf e}_{k}^{\rm H}\varphi_{\ell}(-t T_{k})\beta{\bf e}_{1}\neq 0$, and define
$$
f_{T_{k}}(z)=t\varphi_{\ell}(z)\cdot\frac{{\bf e}_{k}^{\rm H}(zI+t T_{k})^{-1}\beta{\bf e}_{1}}{{\bf e}_{k}^{\rm H}\varphi_{\ell}(-t T_{k})\beta{\bf e}_{1}}.
$$
Then we have
$$
\widehat{\bf e}_{\ell,k}(t)=\frac{1}{2\pi \bf i}{\int_{\widehat{\Gamma}} f_{T_{k}}(z)(zI+tA)^{-1}}dz\cdot\widehat{\bf r}_{\ell,k}(t),
$$
and
\begin{equation}
\|\widehat{\bf e}_{\ell,k}(t)\|_2 \leq \Big\| \frac{1}{2\pi \bf i}{\int_{\widehat{\Gamma}} f_{T_{k}}(z)(zI+tA)^{-1}}dz\Big\|_2\cdot\|\widehat{\bf r}_{\ell,k}(t)\|_2.
\end{equation}
\end{theorem}
\begin{proof}
From the Dunfold-Taylor representation, we have
\begin{equation}\label{equ61}
{\bf y}(t)=\varphi_{\ell}(-tA){\bf v}=\frac{1}{2\pi \bf i}{\int_{\widehat{\Gamma}}\varphi_{\ell}(z)(zI+tA)^{-1}}{\bf v}dz,\quad \ell=0,1,2,\ldots
\end{equation}
and
\begin{equation}\label{equ62}
\widehat{\bf y}_{\ell,k}(t)=V_{k}\varphi_{\ell}(-t T_{k})\beta{\bf e}_{1}=\frac{1}{2\pi \bf i}{\int_{\widehat{\Gamma}}\varphi_{\ell}(z)V_{k}(zI+t T_{k})^{-1}}\beta{\bf e}_{1}dz,\quad \ell=0,1,2,\ldots
\end{equation}
So we have
\begin{eqnarray*}
\widehat{\bf e}_{\ell,k}(t)
&=& {\bf y}(t)-\widehat{\bf y}_{\ell,k}(t)\nonumber \\
&=& \frac{1}{2\pi \bf i}{\int_{\widehat{\Gamma}} \varphi_{\ell}(z)\big[(zI+tA)^{-1}{\bf v}-V_{k}(zI+t T_{k})^{-1}}\beta{\bf e}_{1}\big]dz.
\end{eqnarray*}
Moreover, by (\ref{equ22}) and (\ref{eqn3.3}), we get the following relation
\begin{eqnarray*}
{\bf v}-(zI+tA)V_{k}(zI+t T_{k})^{-1}\beta{\bf e}_{1}&=&t\gamma h_{k+1,k}^{2}\big[{\bf e}_{k}^{\rm H}(zI+t T_{k})^{-1}\beta{\bf e}_{1}\big] V_{k}(I+\gamma H_{k})^{\rm -H}{\bf e}_{k}\\
&&-t h_{k+1,k}\big[{\bf e}_{k}^{\rm H}(zI+t T_{k})^{-1}\beta{\bf e}_{1}\big]{\bf v}_{k+1}.
\end{eqnarray*}
If ${\bf e}_{k}^{\rm H}\varphi_{\ell}(-t T_{k})\beta{\bf e}_{1}\neq 0$, we obtain from (\ref{eqn3.5}) and (\ref{eqnn3.34}) that
\begin{eqnarray}\label{eqn4.11}
\widehat{\bf e}_{\ell,k}(t)
&=& \frac{1}{2\pi \bf i}{\int_{\widehat{\Gamma}} \varphi_{\ell}(z)}(zI+tA)^{-1} \Big[t\gamma h_{k+1,k}^{2}\big[{\bf e}_{k}^{\rm H}(zI+t T_{k})^{-1}\beta{\bf e}_{1}\big] V_{k}(I+\gamma H_{k})^{\rm -H}{\bf e}_{k} \\
&&-t h_{k+1,k}\big[{\bf e}_{k}^{\rm H}(zI+t T_{k})^{-1}\beta{\bf e}_{1}\big]{\bf v}_{k+1}\Big]dz \nonumber \\
&=& \frac{1}{2\pi \bf i}{\int_{\widehat{\Gamma}} \varphi_{\ell}(z)(zI+tA)^{-1}\frac{{\bf e}_{k}^{\rm H}(zI+t T_{k})^{-1}\beta{\bf e}_{1}}{{\bf e}_{k}^{\rm H}\varphi_{\ell}(-t T_{k})\beta{\bf e}_{1}}}\Big[t\gamma h_{k+1,k}^{2}\big[{\bf e}_{k}^{\rm H}\varphi_{\ell}(-t T_{k})\beta{\bf e}_{1}\big] V_{k}(I+\gamma H_{k})^{\rm -H}{\bf e}_{k}\nonumber \\
&&-t h_{k+1,k}\big[{\bf e}_{k}^{\rm H}\varphi_{\ell}(-t T_{k})\beta{\bf e}_{1}\big]{\bf v}_{k+1}\Big]dz \nonumber \\
&=&\frac{1}{2\pi \bf i}{\int_{\widehat{\Gamma}} t\varphi_{\ell}(z)\frac{{\bf e}_{k}^{\rm H}(zI+t T_{k})^{-1}\beta{\bf e}_{1}}{{\bf e}_{k}^{\rm H}\varphi_{\ell}(-t T_{k})\beta{\bf e}_{1}}}\cdot(zI+tA)^{-1}\widehat{\bf r}_{\ell,k}(t)dz.\nonumber
\end{eqnarray}
Let $f_{T_{k}}(z)=t\varphi_{\ell}(z)\frac{{\bf e}_{k}^{\rm H}(zI+t T_{k})^{-1}\beta{\bf e}_{1}}{{\bf e}_{k}^{\rm H}\varphi_{\ell}(-t T_{k})\beta{\bf e}_{1}}$, and notice that
\begin{eqnarray}
\frac{1}{2\pi \bf i}{\int_{\widehat{\Gamma}} t\varphi_{\ell}(z)}\frac{{\bf e}_{k}^{\rm H}(zI+t T_{k})^{-1}\beta{\bf e}_{1}}{{\bf e}_{k}^{\rm H}\varphi_{\ell}(-t T_{k})\beta{\bf e}_{1}}dz&=&\frac{1}{2\pi \bf i}{\int_{\widehat{\Gamma}} f_{T_{k}}(z)dz}=t.
\end{eqnarray}
Therefore, we have from (\ref{eqn4.11}) that
\begin{eqnarray*}\label{equ63}
\|\widehat{\bf e}_{\ell,k}(t)\|_2 &=&\Big\| \frac{1}{2\pi \bf i}{\int_{\widehat{\Gamma}} f_{T_{k}}(z)(zI+tA)^{-1}}dz\cdot\widehat{\bf r}_{\ell,k}(t)\Big\|_2 \nonumber \\
&\leq& \Big\| \frac{1}{2\pi \bf i}{\int_{\widehat{\Gamma}} f_{T_{k}}(z)(zI+tA)^{-1}}dz\Big\|_2\cdot\|\widehat{\bf r}_{\ell,k}(t)\|_2. \nonumber
%&\leq& \|f_{T_{k}}(-t A)\|_2\cdot\|\widehat{\bf r}_{\ell,k}(t)\|_2.
\end{eqnarray*}
\end{proof}

Similar to the Arnoldi approximation, we can give an error estimate of a class of special matrices where the assumption (\ref{equ64}) holds. Let $\widehat{\mu}_{j}=\widehat{a}_{j}+{\bf i}\widehat{b}_{j}$ ($j=1,2,\ldots,k$) be the eigenvalues of matrix $T_{k}$ arranging as $\widehat{\mu}_{1},\ldots,\widehat{\mu}_{k_{2}}$ the real ones and $\widehat{\mu}_{k_{2}+1},\ldots,\widehat{\mu}_{k}$ the
complex conjugate ones. Moreover, let
$$
\widehat{r}_{j}=\big((\varepsilon+\widehat{a}_{j})^{2}+\widehat{b}_{j}^{2}\big)^{1/2},\quad \widehat{R}=\max_{1\leq j \leq k}\widehat{r}_{j},
$$
and
$$\widehat{\omega}_{k}(\varepsilon)=\prod_{j=1}^{k}\big(\widehat{r}_{j}(\varepsilon+\widehat{a}_{j})\big)^{1/2}.$$
Define
$$\widehat{\varsigma}_{k}(\rho)=\prod_{j=1}^{k_{2}}(1+\rho^{2}/\widehat{r}_{j}^{2})^{1/2}\prod_{j=k_{2}+1}^{k}(1+\rho^{2}/\widehat{r}_{j}^{2})^{1/4},$$
and
$$\widehat{d}_{k}(\rho)=\frac{\prod_{j=1}^{k-1}h_{j+1,j}}{\widehat{\omega}_{k}(\varepsilon)\widehat{\varsigma}_{k}(\rho)}.$$
Under these assumptions, we have the following result for the harmonic Arnoldi approximation, whose proof is similar to that of Theorem 4.2.
\begin{theorem}\label{theo4.5}
Let $\varepsilon>0$, and suppose that $k+k_{2}+2\ell\geq 4$ and $W(T_{k}) \subset \Sigma_{\vartheta,a}$. Then for the harmonic Arnoldi approximation, we have
\begin{equation}\label{equ64-2}
\|{\bf \widehat{e}}_{\ell,k}(t)\|_{2}\leq \widehat{c}_{\ell,k}\cdot\int_{0}^{\infty}\frac{(1+\rho^{2}/\varepsilon^{2})^{-\ell/2}}{\widehat{\varsigma}_{k}(\rho)}d\rho\cdot\|{\bf \widehat{r}}_{\ell,k}(t)\|_{2}
\end{equation}
where $\widehat{c}_{\ell,k}=\frac{\exp(t\varepsilon)}{\pi (t\varepsilon)^{\ell}\widehat{\omega}_{k}(\varepsilon)(\varepsilon+a)}\cdot\frac{\prod_{j=1}^{k-1}h_{j+1,j}}{|{\bf e}_{k}^{\rm H}\varphi_{\ell}(-t T_{k}){\bf e}_{1}|}$.
\end{theorem}

The following result avoids the use of a quadrature rule for evaluating the integral in (\ref{equ64-2}). The proof is analogous to that of Theorem \ref{theo4.3}.
\begin{theorem}\label{theo4.6}
Let $\varepsilon>0$ and suppose that $k+k_{2}\geq 4$. Under the above notations, we have that
$$\|{\bf \widehat{e}}_{\ell,k}(t)\|_{2}\leq \widehat{c}_{\ell,k}\cdot \widehat{C}_{k}\cdot\|{\bf \widehat{r}}_{\ell,k}(t)\|_{2},$$
where
$$\widehat{C}_{k}=\frac{\sqrt{\pi}}{2\sqrt{\widehat{S}_{1}}}+\exp\big(-\varepsilon^{2}\widehat{S}_{2}\big)(\widehat{R}-\varepsilon)+\Big(\frac{\varepsilon}
{\sqrt{\varepsilon^{2}+\widehat{R}^{2}}}\Big)^{\ell}\frac{\pi \widehat{R}}{2^{(\frac{k+k_{2}}{4}+1)}},$$
with
$$\widehat{S}_{1}=\frac{\ell}{4\varepsilon^{2}}+\frac{1}{2}\sum_{j=1}^{k_{2}}\frac{1}{\widehat{r}_{j}^{2}+\varepsilon^{2}}+\frac{1}{4}
\sum_{j=k_{2}+1}^{k}\frac{1}{\widehat{r}_{j}^{2}+\varepsilon^{2}},$$
and
$$\widehat{S}_{2}=\frac{\ell}{2(\varepsilon^{2}+\widehat{R}^{2})}+\frac{1}{2}\sum_{j=1}^{k_{2}}\frac{1}{\widehat{r}_{j}^{2}+\widehat{R}^{2}}
+\frac{1}{4}\sum_{j=k_{2}+1}^{k}\frac{1}{\widehat{r}_{j}^{2}+\widehat{R}^{2}}.$$
\end{theorem}

%\begin{rem}
%Theorems 4.1 and 4.2 reveal that a small residual does not necessarily imply a small error. However, if $\|\frac{1}{2\pi \bf i}{\int_{{\Gamma}} f_{H_{k}}(z)(zI+tA)^{-1}}dz\|_2$ and $\| \frac{1}{2\pi \bf i}{\int_{\widehat{\Gamma}} f_{T_{k}}(z)(zI+tA)^{-1}}dz\|_2$ is of medium size, then the residual norm can be used as a reliable stopping criterion in the thick-restarted Arnoldi algorithm and the harmonic thick-restarted Arnoldi algorithm.
%\end{rem}

\section{The advantage of the thick-restarting strategy for matrix functions}

\setcounter{equation}{0}

In this section, we show the advantage of augmenting approximate eigenvectors in the thick-restarted Arnoldi and the harmonic Arnoldi algorithms.
For simplicity, we consider the case of augmenting only one (approximate) eigenvector.
Let's first discuss an ``ideal" case in which an ``exact" eigenvector ${\bf x}$ is added into the search space.
The orthonormal basis is $V_{k}=[{\bf x},{\bf v}_{2},{\bf v}_{3},\ldots,{\bf v}_{k}]=[{\bf x},~V_{\bot}]$ after restarting, where $V_{\bot}=[{\bf v}_{2},{\bf v}_{3},\ldots,{\bf v}_{k}]$ and $A{\bf x}=\lambda{\bf x}$. Then we have
\begin{equation}\label{eqn5.1}
H_{k}=V_{k}^{\rm H}AV_{k}=\left[\begin{array}{c} {\bf x}^{\rm H}\\
V_{\bot}^{\rm H}
\end{array}\right]\left[\begin{array}{cc} A{\bf x} & AV_{\bot}\\
\end{array}\right]=\left[\begin{array}{cc} \lambda & {\bf x}^{\rm H}AV_{\bot}\\
{\bf 0} & V_{\bot}^{\rm H}AV_{\bot}
\end{array}\right]\equiv\left[\begin{array}{cc} \lambda & H_{12}\\
{\bf 0} & H_{22}
\end{array}\right],
\end{equation}
%%%%%%%%%%%%%%%%%%%%%%%%%%%%%%%%%%%%%%%%%%%%%%%%%%%%%%%%%%%%%%%
%and we proof the following fact:
%\begin{equation}\label{equ64}
%H_{k}^{m}=\left[\begin{array}{cc} \lambda^{m} & {\bf h}^H\\
%0 & H_{22}^{m}
%\end{array}\right]
%\end{equation}
%where ${\bf h}\in\mathbb{C}^{k-1}$ is a vector, $m \geq 1$.
%Because of
%\begin{equation*}
%\varphi_{0}(-t H_{k})\beta{\bf e}_{1}=\beta \left[\begin{array}{cc} \varphi_{0}(-t \lambda)\\
%0
%\end{array}\right]
%\end{equation*}
%we assume $\varphi_{\ell-1}(-t H_{k})\beta{\bf e}_{1}=\beta \left[\begin{array}{cc} \varphi_{\ell-1}(-t \lambda)\\
%0
%\end{array}\right]$. By the recurrence relation $\varphi_{\ell-1}(z)=z\varphi_{\ell}(z)+\frac{1}{(\ell-1)!}, \quad \ell=1,2,\ldots$,
%we have
%\begin{equation*}
%\varphi_{\ell-1}(-t H_{k})\beta{\bf e}_{1}=-t H_{k} \varphi_{\ell}(-t H_{k})\beta{\bf e}_{1}+\frac{1}{(\ell-1)!}\beta{\bf e}_{1}
%\end{equation*}
%and we have
%\begin{equation}\label{equ65}
%-t H_{k} \varphi_{\ell}(-t H_{k})\beta{\bf e}_{1}=\beta \left[\begin{array}{cc} \varphi_{\ell-1}(-t \lambda)-\frac{1}{(\ell-1)!}\\
%0
%\end{array}\right]
%\end{equation}
%When we assume the matrix $H_{k}$ is reversible, we have the following equation:
%\begin{eqnarray}\label{equ66}
%\varphi_{\ell}(-t H_{k})\beta{\bf e}_{1}
%&=& (-t H_{k})^{-1}\beta \left[\begin{array}{cc} \varphi_{\ell-1}(-t \lambda)-\frac{1}{(\ell-1)!}\\
%0
%\end{array}\right]\nonumber \\
%&=& \left[\begin{array}{cc} \beta \varphi_{\ell}(-t \lambda)\\
%0
%\end{array}\right]\
%\end{eqnarray}
%%%%%%%%%%%%%%%%%%%%%%%%%%%%%%%%%%%%%%%%%%%%%%%%%%%%%%%%%%%%%%%%%%%%%%%%%%%%%%%%%%%%%
and
\begin{equation*}
(zI+t H_{k})^{-1}=\left[\begin{array}{cc} (z+t \lambda)^{-1} & {\bf s}^{\rm H}\\
{\bf 0} & (zI+t V_{\bot}^{\rm H} A V_{\bot})^{-1}
\end{array}\right],
\end{equation*}
where ${\bf s}\in\mathbb{C}^{k-1}$ is a vector. As a result,
\begin{equation*}
{\bf e}_{k}^{\rm H}(zI+t H_{k})^{-1}\beta{\bf e}_{1}=0,
\end{equation*}
so we have from (\ref{equ66-2}) that
\begin{equation}\label{equ6-3}
{\bf e}_{\ell,k}(t)={\bf y}(t)-{\bf y}_{\ell,k}(t)=\frac{1}{2\pi \bf i}{\int_\Gamma \varphi_{\ell}(z)}(zI+tA)^{-1} \Big[-t h_{k+1,k}\big[{\bf e}_{k}^{\rm H}(zI+t H_{k})^{-1}\beta{\bf e}_{1}\big]{\bf v}_{k+1}\Big]dz={\bf 0}.
\end{equation}

Now let's consider the harmonic Arnoldi approximation. If $V_{k}=[{\bf x},{\bf v}_{2},{\bf v}_{3},\ldots,{\bf v}_{k}]$, we notice from
(\ref{eqn5.1}) and (\ref{eqn3.3}) that
\begin{equation}
{\bf e}_{k}^{\rm H}(zI+t T_{k})^{-1}\beta{\bf e}_{1}
={\bf e}_{k}^{\rm H}(zI+t H_{k})^{-1}\beta{\bf e}_1=0.
\end{equation}
Consequently, we have from (\ref{eqn4.11}) that
\begin{eqnarray}\label{equ6-22}
\widehat{\bf e}_{\ell,k}(t)
&=& {\bf y}(t)-\widehat{\bf y}_{\ell,k}(t)\nonumber \\
&=& \frac{1}{2\pi \bf i} {\int_{\widehat{\Gamma}} \varphi_{\ell}(z)(zI+tA)^{-1}}\Big[t\gamma h_{k+1,k}^{2} \big[{\bf e}_{k}^{\rm H}(zI+t T_{k})^{-1}\beta{\bf e}_{1}\big] V_{k}(I+\gamma H_{k})^{-1}{\bf e}_{k}\nonumber\\
&& -t h_{k+1,k} \big[{\bf e}_{k}^{\rm H}(zI+t T_{k})^{-1}\beta{\bf e}_{1}\big] {\bf v}_{k+1} \Big]dz \nonumber\\
&=&{\bf 0}.
\end{eqnarray}
\begin{rem}
Equations {\rm(}\ref{equ6-3}{\rm)} and {\rm(}\ref{equ6-22}{\rm)} indicate that, if the search subspace
is augmented with an exact eigenvector in the thick-restarted Arnoldi and harmonic Arnoldi algorithms, then
we will get the exact solution.
\end{rem}

In practical calculations, we are interested in the situation where an approximate eigenvector (say, the Ritz vector or the harmonic Ritz vector) $\widetilde{\bf x}$ is added into
the search space spanned by
$$
V_{k}=[\widetilde{\bf x},{\bf v}_{2},{\bf v}_{3},\ldots,{\bf v}_{k}]\equiv[\widetilde{\bf x},~V_{\bot}].
$$
Let the residual of the Ritz pair $(\widetilde{\lambda},\widetilde{\bf x})$ be $\widetilde{\bf r}=A \widetilde{\bf x}-\widetilde\lambda \widetilde{\bf x}$, we have
\begin{equation*}
H_{k}=V_{k}^{\rm H}AV_{k}=\left[\begin{array}{c} \widetilde{\bf x}^{\rm H}\\
V_{\bot}^{\rm H}
\end{array}\right]\left[\begin{array}{cc} A\widetilde{\bf x} & A V_{\bot}\\
\end{array}\right]=\left[\begin{array}{cc} \mu & \widetilde{\bf x}^{\rm H}A V_{\bot}\\
V_{\bot}^{\rm H}\widetilde{\bf r} & V_{\bot}^{\rm H}AV_{\bot}
\end{array}\right]\equiv\left[\begin{array}{cc} \mu & \widetilde H_{12}\\
V_{\bot}^{\rm H}\widetilde{\bf r} & H_{22}
\end{array}\right],
\end{equation*}
where we used $V^{\rm H}_{\bot}A\widetilde{\bf x}=V^{\rm H}_{\bot}(\widetilde{\bf r}+\widetilde{\lambda}\widetilde{\bf x})=V^{\rm H}_{\bot}\widetilde{\bf r}$, and $\mu=\widetilde{\bf x}^{\rm H}A\widetilde{\bf x}$ is a scalar. As only one approximate eigenvector is added into $V_k$, we see that
$H_k$ is still an upper Hessenberg matrix.
Denote by $h_{2,1}$ the $(2,1)$ element of $H_{k}$, then $V^{\rm H}_{\bot}\widetilde{\bf r}=[h_{2,1},0,\ldots,0]^{\rm H}$ and
\begin{equation}\label{equ5.5}
|h_{2,1}|=\|V_{\bot}^{\rm H}\widetilde{\bf r}\|_2\leq\|\widetilde{\bf r} \|_{2}.
\end{equation}
Let
\begin{equation}
\widetilde{Q}=(zI+t H_{k})-t h_{2,1}{\bf e}_{2}{\bf e}_{1}^{\rm H},
\end{equation}
note that $\widetilde Q$ is an upper Hessenberg matrix with its $(2,1)$ element being zero. If $\widetilde{Q}$ is nonsingular and $t h_{2,1}({\bf e}_{1}^{\rm H}\widetilde Q^{-1}{\bf e}_{2})\neq -1$, by the Sherman--Morrison formula \cite{GV}, we obtain
\begin{eqnarray}\label{eqn5.6}
{\bf e}_{k}^{\rm H}(zI+t H_{k})^{-1}\beta{\bf e}_{1}&=&{\bf e}_{k}^{\rm H}(\widetilde Q+ t h_{2,1}{\bf e}_{2}{\bf e}_{1}^{\rm H})^{-1}\beta{\bf e}_{1}\nonumber\\
&=&-\frac{t h_{2,1}\beta(z+t\mu)^{-1}}{1+t h_{2,1}({\bf e}_{1}^{\rm H}\widetilde Q^{-1}{\bf e}_{2})}\cdot ({\bf e}_{k}^{\rm H}\widetilde Q^{-1}{\bf e}_{2})\nonumber\\
&\equiv& h_{2,1}\cdot \xi(z),
\end{eqnarray}
where
\begin{equation}
\xi(z)=-\frac{t\beta(z+t\mu)^{-1}}{1+t h_{2,1}({\bf e}_{1}^{\rm H}\widetilde Q^{-1}{\bf e}_{2})}\cdot ({\bf e}_{k}^{\rm H}\widetilde Q^{-1}{\bf e}_{2}).
\end{equation}
So we have
\begin{eqnarray}\label{equ5.9}
\|{\bf e}_{\ell,k}(t)\|_2
&=& \Big\|\frac{1}{2\pi \bf i}{\int_\Gamma \varphi_{\ell}(z)}(zI+tA)^{-1} \Big[-t h_{k+1,k}\big[{\bf e}_{k}^{\rm H}(zI+t H_{k})^{-1}\beta{\bf e}_{1}\big]{\bf v}_{k+1}\Big]dz\Big\| _2 \nonumber\\
&=& |t h_{k+1,k}h_{2,1}|\cdot\bigg\|\frac{1}{2\pi \bf i}{\int_\Gamma \varphi_{\ell}(z)}(zI+tA)^{-1}\cdot\xi(z){\bf v}_{k+1}dz\bigg\|_2.
\end{eqnarray}
By (\ref{equ5.9}) and (\ref{equ5.5}), we have the following theorem. It shows the advantage of augmenting a Ritz vector in
the search space of the thick-restarted Arnoldi method.
\begin{theorem}
Let $(\widetilde{\lambda},\widetilde{\bf x})$ be a Ritz pair with residual $\widetilde{\bf r}$, and let $V_{k}=[\widetilde{\bf x},{\bf v}_{2},{\bf v}_{3},\ldots,{\bf v}_{k}]$.
Then in the thick-restarted Arnoldi algorithm, we have
\begin{equation}\label{equ6-6}
\|{\bf e}_{\ell,k}(t)\|_2 \leq \|\widetilde{\bf r}\|_2\cdot |t h_{k+1,k}|\bigg\|\frac{1}{2\pi \bf i}{\int_\Gamma \varphi_{\ell}(z)}(zI+tA)^{-1}\cdot\xi(z){\bf v}_{k+1}dz\bigg\|_2.
\end{equation}
\end{theorem}

Next we consider the thick-restarted harmonic Arnoldi algorithm. For simplicity, we still denote
$V_{k}=[\widetilde{\bf x},{\bf v}_{2},{\bf v}_{3},\ldots,{\bf v}_{k}]$,
whose columns span the search subspace; and let the residual of the harmonic Ritz pair $(\widetilde{\lambda},\widetilde{\bf x})$ be $\widetilde{\bf r}=A \widetilde{\bf x}-\widetilde\lambda \widetilde{\bf x}$.
If $1+t\gamma h_{k+1,k}^{2}[{\bf e}_{k}^{\rm H}(zI+t H_{k})^{-1}(I+\gamma H_{k})^{\rm -H}{\bf e}_{k}]\neq 0$, we obtain from (\ref{eqn3.3}) and the Sherman--Morrison formula that
\begin{eqnarray}\label{eqn5.11}
{\bf e}_{k}^{\rm H}(zI+t T_{k})^{-1}\beta{\bf e}_{1}
&=&{\bf e}_{k}^{\rm H}\Big[(zI+t H_{k})^{-1}\beta{\bf e}_{1}-\big[1+t\gamma h_{k+1,k}^{2}[{\bf e}_{k}^{\rm H}(zI+t H_{k})^{-1}(I+\gamma H_{k})^{\rm -H}{\bf e}_{k}]\big]^{-1}\nonumber\\
&& (zI+t H_{k})^{-1}t\gamma h_{k+1,k}^{2}(I+\gamma H_{k})^{\rm -H}{\bf e}_{k}\big[{\bf e}_{k}^{\rm H}(zI+t H_{k})^{-1}\beta{\bf e}_{1}\big]\Big]\nonumber\\
&=&\big[1-\zeta_1(z)\zeta_2(z)\big][{\bf e}_{k}^{\rm H}(zI+t H_{k})^{-1}\beta{\bf e}_{1}],
\end{eqnarray}
where
$$
\zeta_1(z)=\big[1+t\gamma h_{k+1,k}^{2}[{\bf e}_{k}^{\rm H}(zI+t H_{k})^{-1}(I+\gamma H_{k})^{\rm -H}{\bf e}_{k}]\big]^{-1},
$$
and
$$
\zeta_2(z)=
t\gamma h_{k+1,k}^{2}\big[{\bf e}_k^{\rm H}(zI+t H_{k})^{-1}(I+\gamma H_{k})^{\rm -H}{\bf e}_{k}\big].
$$
Therefore, we have from (\ref{eqn5.6}) and (\ref{eqn5.11}) that
\begin{equation}\label{eqn5.12}
{\bf e}_{k}^{\rm H}(zI+t T_{k})^{-1}\beta{\bf e}_{1}=\big[1-\zeta_1(z)\zeta_2(z)\big]\xi(z)h_{2,1}\equiv \chi(z)h_{2,1},
\end{equation}
and
\begin{equation*}
\big|{\bf e}_{k}^{\rm H}(zI+t T_{k})^{-1}\beta{\bf e}_{1}\big|\leq|\chi(z)|\cdot\|\widetilde{\bf r}\|_2.
\end{equation*}
Denote
$$
{\bf c}_{k+1}(z)=\left[\begin{array}{c} \gamma h_{k+1,k}\chi(z)(I+\gamma H_{k})^{\rm -H}{\bf e}_{k}\\
-\chi(z)
\end{array}\right],
$$
from the relations (\ref{eqn4.11}) and (\ref{eqn5.12}), we get
\begin{eqnarray}\label{equ5.13}
\|\widehat{\bf e}_{\ell,k}(t)\|_2
&=& \Big\|\frac{1}{2\pi \bf i} {\int_{\widehat{\Gamma}} \varphi_{\ell}(z)(zI+tA)^{-1}}\Big[t\gamma h_{k+1,k}^{2} \big[{\bf e}_{k}^{\rm H}(zI+t T_{k})^{-1}\beta{\bf e}_{1}\big] V_{k}(I+\gamma H_{k})^{-1}{\bf e}_{k}\nonumber \\
&& -t h_{k+1,k} \big[{\bf e}_{k}^{\rm H}(zI+t T_{k})^{-1}\beta{\bf e}_{1}\big] {\bf v}_{k+1} \Big]dz\Big\|_2 \nonumber \\
&=&\Big\|\frac{1}{2\pi \bf i} {\int_{\widehat{\Gamma}} \varphi_{\ell}(z)(zI+tA)^{-1}}\Big[t\gamma h_{k+1,k}^{2} \big[\chi(z)h_{2,1}\big] V_{k}(I+\gamma H_{k})^{\rm -H}{\bf e}_{k}\nonumber \\
&& -t h_{k+1,k} \big[\chi(z)h_{2,1}\big]{\bf v}_{k+1} \Big]dz\Big\|_2 \nonumber \\
&=& \Big\|t h_{k+1,k}h_{2,1}\cdot\frac{1}{2\pi \bf i} {\int_{\widehat{\Gamma}} \varphi_{\ell}(z)(zI+tA)^{-1}}\cdot V_{k+1}{\bf c}_{k+1}(z)dz\Big\|_2.
\end{eqnarray}
From (\ref{equ5.13}) and (\ref{equ5.5}), we obtain the following theorem. It shows the merit of augmenting a harmonic Ritz vector in
the search space of the thick-restarted harmonic Arnoldi method.
\begin{theorem}
Let $(\widetilde{\lambda},\widetilde{\bf x})$ be a harmonic Ritz pair with residual $\widetilde{\bf r}$, and let $V_{k}=[\widetilde{\bf x},{\bf v}_{2},{\bf v}_{3},\ldots,{\bf v}_{k}]$. Then in the thick-restarted harmonic Arnoldi algorithm, we have
\begin{equation}\label{equ6-6}
\|\widehat{\bf e}_{\ell,k}(t)\|_2 \leq \|\widetilde{\bf r}\|_2\cdot |t h_{k+1,k}|\bigg\|\frac{1}{2\pi \bf i}{\int_{\widehat{\Gamma}} \varphi_{\ell}(z)}(zI+tA)^{-1}\cdot V_{k+1}{\bf c}_{k+1}(z)dz\bigg\|_2.
\end{equation}
\end{theorem}

\section{Numerical experiments}

\setcounter{equation}{0}

In this section, we make some numerical experiments to show the
superiority of our new algorithm over many state-of-the-art algorithms for computing $\varphi$-functions. The numerical experiments are run
on a Dell PC with eight core Intel(R) Core(TM)i7-2600 processor with CPU
3.40 GHz and RAM 16.0 GB, under the Windows 7 with 64-bit operating system. All the numerical results are obtained from using a MATLAB 7.10.0 implementation
with machine precision $\epsilon\approx 2.22\times 10^{-16}$.
The algorithms used in this section are listed as follows.

$\bullet$ {\bf phipm} \cite{NW2} computes the action of
linear combinations of $\varphi$-functions on operand vectors. The implementation combines
time stepping with a procedure to adapt the Krylov subspace size. The MATLAB codes are
available from {\it http://www1.maths.leeds.ac.uk/\textasciitilde jitse/software.html}.

$\bullet$ {\bf expv} is the
MATLAB function due to Sidje \cite{Sidje}, which evaluates ${\rm exp}(-tA){\bf v}$ using a restarted Krylov subspace method with
a fixed dimension. The MATLAB codes are available from {\it http://www.maths.uq.edu. au/expokit/}.

$\bullet$ ${\bf funm_-kryl}$ is a realization of the Krylov subspace method
with deflated restarting for matrix functions \cite{EEG}. Its effect is to ultimately deflate a specific invariant subspace of the matrix which most impedes the convergence of the restarted Arnoldi approximation process. The MATLAB codes are available from {\it http://www.mathe.tu-freiberg.de/\~~guettels/funm\_kryl/}.

$\bullet$ ${\bf funm_-quad}$ is a realization of
the restarted Arnoldi algorithm described in \cite{From}. This algorithm utilizes an integral representation for the error of the iterates in the Arnoldi
method which then allows one to develop a quadrature-based restarting algorithm suitable
for a large class of functions. It can be viewed as an improved version of the deflated restarting Krylov algorithm proposed in \cite{EEG}. The MATLAB codes can be downloaded
from {\it http://www.guettel.com/$funm_-quad$}.

$\bullet$ ${\bf Rich_-Kryl}$ is the restarted and residual-based Krylov-Richardson algorithm for computing the matrix exponential problem ${\rm exp}(-tA){\bf v}$ \cite{Residual}.

$\bullet$ {\bf TRA} and {\bf TRHA} are the thick-restarted Arnoldi algorithm and the thick-restarted harmonic Arnoldi algorithm (Algorithm 1), respectively, for
evaluating $\varphi_{\ell}(-tA){\bf v}$, $\ell=0,1,\ldots,s$.

\vspace{0.2cm}

We run the MATLAB functions {\it phipm, expv, ${\it funm_-kryl}$} and ${\it funm_-quad}$ using
their default parameters. In all the algorithms, the convergence tolerance for $\varphi$-functions is chosen as $tol=10^{-8}$, and the dimension $k$ for the Krylov subspace is set to be 30.
In the deflated Krylov subspace algorithms ${\it funm_-kryl}$, ${\it funm_-quad}$, TRA and TRHA,
we set the number $q$
of approximate eigenvectors retained from the previous cycles to be $5$, and augment the search subspace with approximate eigenvectors corresponding to the smallest approximate eigenvalues. The parameter $\gamma$ in TRHA is set to be $\gamma=0.01t$ in all the numerical examples. For the reduced matrices (projection matrices), the matrix exponential are computed by using the MATLAB built-in function {\it expm}, and the $\varphi_\ell~(\ell\geq 1)$ functions are computed by using the {\it phipade} function of the EXPINT package available from {\it http://www.math.ntnu.no/num/expint/}.

In the residual-based algorithms ${\rm Rich_-Kryl}$, TRA and TRHA,
we solve the initial value problems by using the {\it ode15s} ODE solver in MATLAB, whose absolute and relative tolerances are chosen as $10^{-9}$.
%{\bf In order to solve the small sized ODE (\ref{eqn3.41}), we have to store some values of
%right-hand side for different $t$ spanning the time interval of interest. In practical experiments, we find that (\ref{eqn3.41}) is difficult to be solved numerically when $t$ approaches to zero in the time interval of interest. To avoid this difficulty, we can solve the problems (\ref{6.2}) below for matrix functions $t^{\ell}\varphi_\ell(-t A){\bf v}$ in the same method for computing (\ref{equ49}). Let $t=1$, we can get $\varphi_\ell(-A){\bf v}$. If $t\neq 1$, we denote $A$ by $tA$.}
In the tables below,
we denote by ``CPU" the CPU time in seconds, and by ``Mv" the number of
matrix-vector products. Let ${\bf y}(t)$ be the ``exact" solution, and let $\widetilde{\bf y}(t)$ be
an approximation obtained from running the above algorithms,
then we define the relative error
$$
{\bf Error}=\frac{\|{\bf y}(t)-\widetilde{\bf y}(t)\|_{2}}{\|{\bf y}(t)\|_{2}}.
$$
If an algorithm does not converge
within an acceptable CPU timing (say, $6$ hours), then we stop and declare that the algorithm ``fails to converge".

\vspace{0.2cm}\textbf{Example 6.1.}~In this example, we compare TRHA with {\it phipm}, $\it{funm_-kryl}$ and TRA for
the computation of $\varphi$-functions, and show the efficiency of our new algorithm for solving (\ref{eqn1.3}) {\it simultaneously}.
The test problem is routinely used to study performance of stiff integrators \cite{Hair,Tok}.
Consider the following two-dimensional semilinear reaction-diffusion-advection equation
\begin{equation}\label{eqn6.5}
u_{t}=\varepsilon_{1}(u_{xx}+u_{yy})-\beta_{1}(u_{x}+u_{y})+\rho_{1} u\big(u-\frac{1}{2}(1-u)\big)
\end{equation}
defined on the unit square $\Omega=[0,1]^{2}$, which satisfies the homogeneous Dirichlet boundary conditions. We set $\varepsilon_{1}=0.02$, $\beta_{1}=-0.02$, $\rho_{1}=1$,
and use
$$
u(t=0,x,y)=256(xy(1-x)(1-y))^{2}+0.3
$$
as the initial condition.

\begin{center}
\begin{tabular}{|c|c|c|c|c|}
\multicolumn{4}{c}{{}}\\[3pt]\hline
~$\ell$~&~{\bf Algorithm}~&~{\bf CPU}~&~{\bf Error}~&~{\bf Mv}\\
\hline & phipm  &151.11 &$4.757\times 10^{-13}$  &3389\\
\cline{2-5} & $\rm funm_-kryl$  &168.03 &$8.000\times 10^{-12}$  &1355\\
\cline{2-5}$0$ & TRA &178.22 &$5.277\times 10^{-9}$  &1780\\ \cline{2-5}
& TRHA &163.17 &$1.413\times 10^{-8}$  &1630\\  \hline
\hline & phipm  &163.61 &$2.069\times 10^{-13}$  &2408\\
\cline{2-5} & $\rm funm_-kryl$  &221.22 &$7.116\times 10^{-12}$  &1255\\
\cline{2-5}$1$& TRA &159.58 &$1.513\times 10^{-9}$  &1655\\ \cline{2-5}
& TRHA &145.71 &$1.179\times 10^{-9}$  &1505\\  \hline
\hline & phipm  &148.03 &$4.198\times 10^{-13}$  &2149\\
\cline{2-5} & $\rm funm_-kryl$  &249.28 &$1.347\times 10^{-11}$  &1155\\
\cline{2-5}$2$& TRA &145.02 &$3.007\times 10^{-9}$  &1530\\ \cline{2-5}
& TRHA &131.26 &$1.703\times 10^{-9}$  &1380\\\hline
\hline & phipm  &138.88 &$9.892\times 10^{-14}$  &2035\\
\cline{2-5} & $\rm funm_-kryl$  &236.98 &$3.703\times 10^{-11}$  &1055\\
\cline{2-5}$3$& TRA &133.23 &$8.363\times 10^{-9}$  &1405\\ \cline{2-5}
& TRHA &119.95 &$4.685\times 10^{-9}$  &1280\\\hline
\hline & phipm  &601.63 &--  &9981\\
\cline{2-5}& $\rm funm_-kryl$  &875.50 &--  &4820\\
\cline{2-5}total& TRA &616.04 &--  &6370\\ \cline{2-5}
& TRHA &560.08 &--  &5795\\\hline
\end{tabular}\\[5pt]
\end{center}
{Example 6.1.~Table 1: Numerical results of the 2D reaction-diffusion-advection equation (\ref{eqn6.5}), the matrix size $n=N^2=250,000$. Compute $\varphi_\ell(-tA){\bf u}_0,~\ell=0,1,2,3$ {\it sequentially} (one by one) by
using TRA, TRHA, {\it phipm} and $\it funm_-kryl$.}\\

We discretize (\ref{eqn6.5}) spatially by standard finite differences with meshwidth $\Delta x=\Delta y=\frac{1}{N+1}$ and $N=500$.
This gives a system of ODEs of size $N^{2}$:
$$\mathbf{u}'(t)=-A\mathbf{u}(t)+\mathbf{f}(\mathbf{u}),~~\mathbf{u}(0)=\mathbf{u}_{0}.$$
This linear differential system can be efficiently solved by means of the exponential Runge-Kutta integrators \cite{Hochbruck-Ostermann-2010}.
More precisely, $\mathbf{u}(t_{\hat{n}+1})$ can be approximated from $t_{\hat{n}}$ to $t_{\hat{n}+1}=t_{\hat{n}}+\Delta t~(\hat{n}=0,1,2,\ldots)$  by $\mathbf{u}_{\hat{n}+1}$ defined as
$$\mathbf{u}_{\hat{n}+1}=\mathbf{u}_{\hat{n}}+\Delta t\sum_{i=1}^{\hat{s}}c_{i}(-\Delta tA)(\mathbf{f}_{\hat{n}i}-A\mathbf{u}_{\hat{n}}),$$
where
$$\mathbf{f}_{\hat{n}i}=\mathbf{f}(\mathbf{U}_{\hat{n}i}),~~ i=1,\ldots,\hat{s},$$
with
$$\mathbf{U}_{\hat{n}i}=\mathbf{u}_{\hat{n}}+\Delta t\sum_{j=1}^{i-1}a_{ij}(-\Delta tA)(\mathbf{f}_{\hat{n}j}-A\mathbf{u}_{\hat{n}}),$$
and the coefficients $c_{i}$, $a_{ij}$ are constructed from the $\varphi$-functions.
If Krogstad's four-stage scheme (see \cite{SK-2005} and Example 2.19 of \cite{Hochbruck-Ostermann-2010}) is used to integrate the system of ODEs,
one needs to compute the terms $\varphi_{\ell}(-\Delta tA)(\mathbf{f}_{\hat{n}1}-A\mathbf{u}_{\hat{n}})$ with $\ell=1,2,3$ simultaneously in each time step. Similarly,
if the generalized Lawson scheme (see Example 2.34 of \cite{Hochbruck-Ostermann-2010}) is used, the vectors $\varphi_{\ell}(-\frac{\Delta t}{2}A)\mathbf{\hat{v}}$  with $\ell=0,1,2$
are necessary to be approximated for the same vector $\mathbf{\hat{v}}$ in each time step.

In this example, we want to compute $\varphi_{\ell}(-tA)\mathbf{u_{0}}$ with $\ell=0,1,2,3$ and $t=1$. Here the ``exact" solutions
are obtained from running the MATLAB function {\it phipm} with convergence tolerance $tol=10^{-14}$.
In Table 1, we list the CPU time and the number of matrix-vector products for computing the four vectors {\it sequentially} (one by one); while in Table 2, we present those for evaluating the four vectors {\it simultaneously}. So as to illustrate the merit of TRHA for solving the four vectors simultaneously, in Table 1, we also list the total CPU time and the total number of matrix-vector products for computing the four vectors sequentially.

\begin{center}
\begin{tabular}{|c|c|c|c|c|}
\multicolumn{4}{c}{{}}\\[3pt]\hline
~$\ell$~&~{\bf Algorithm}~&~{\bf CPU}~&~{\bf Error}~&~{\bf Mv}\\
\hline  &   & &$8.000\times 10^{-12}$  &\\
$0\sim3$& $\rm funm_-kryl$ &663.25 &$2.454\times 10^{-13}$  &1355\\
&  & &$1.171\times 10^{-13}$  &\\
&  & &$8.952\times 10^{-14}$  &\\\hline
\hline  &   & &$5.277\times 10^{-9}$  &\\
$0\sim3$& TRA &204.91 &$1.513\times 10^{-9}$  &1780\\
&  & &$3.007\times 10^{-9}$  &\\
&  & &$8.363\times 10^{-9}$  &\\\hline
\hline  &   & &$1.413\times 10^{-8}$  &\\
$0\sim3$& TRHA &191.14 &$1.180\times 10^{-9}$  &1630\\
&  & &$1.703\times 10^{-9}$  &\\
&  & &$4.685\times 10^{-9}$  &\\\hline
\end{tabular}\\[5pt]
\end{center}
{Example 6.1.~Table 2: Numerical results of the 2D reaction-diffusion-advection equation (\ref{eqn6.5}), the matrix size $n=N^2=250,000$. Compute $\varphi_\ell(-tA){\bf u}_0,~\ell=0,1,2,3$ {\it simultaneously} by using TRA, TRHA and $\it funm_-kryl$.}\\

Three remarks are in order. First, we observe from Table 1 and Table 2 that, whether this problem is solved sequentially or simultaneously, TRHA always works better than the other algorithms
in terms of CPU time, especially when $\ell$ is large. Second, it is seen that when this problem is solved simultaneously, the total CPU time of TRHA are much less than those of {\it phipm} and $\it funm_-kryl$. More precisely, we can save about $\frac{2}{3}$ CPU time, 191.14 seconds vs. 601.63 and 663.25 seconds. However, we find that the number of matrix-vector products of {\it phipm} is less than those of TRA and TRHA. Indeed, solving small ODE problems during cycles is a large overhead for TRA and TRHA, and the number of matrix-vector products is not the whole story for accessing the computational complexities.
Third, the accuracy of the approximations obtained from {\it phipm} and $\it funm_-kryl$
can be (much) higher than that obtained from TRA and TRHA. The reason is due to the fact that we need to solve a small-sized ODE problem during each cycle of the residual based algorithms.

\begin{center}
\begin{tabular}{|c|c|c|c|c|}
\multicolumn{4}{c}{{}}\\[3pt]\hline
~$\ell$~&~{\bf Algorithm}~&~{\bf CPU}~&~{\bf Error}~&~{\bf Mv}\\
\hline & phipm  &171.07 &$1.356\times 10^{-14}$  &2477\\
\cline{2-5} & $\rm funm_-kryl$  &198.80 &$1.338\times 10^{-11}$  &1205\\
\cline{2-5}$1$ & TRA &147.50 &$1.946\times 10^{-10}$  &1555\\ \cline{2-5}
& TRHA &137.20 &$1.122\times 10^{-10}$  &1455\\  \hline
\hline & phipm  &154.43 &$2.977\times 10^{-13}$  &2239\\
\cline{2-5} & $\rm funm_-kryl$  &185.29 &$3.529\times 10^{-11}$  &1055\\
\cline{2-5}$2$& TRA &133.55 &$3.124\times 10^{-10}$  &1430\\ \cline{2-5}
& TRHA &121.45 &$1.898\times 10^{-10}$  &1305\\  \hline
\hline & phipm  &135.28 &$2.564\times 10^{-12}$  &2448\\
\cline{2-5} & $\rm funm_-kryl$  &161.12 &$7.971\times 10^{-11}$  &930\\
\cline{2-5}$3$& TRA &119.02 &$1.520\times 10^{-9}$  &1280\\ \cline{2-5}
& TRHA &107.43 &$9.209\times 10^{-10}$  &1155\\\hline
\hline & phipm  &137.17 &$4.843\times 10^{-13}$  &2844\\
\cline{2-5} & $\rm funm_-kryl$  &124.83 &$2.715\times 10^{-10}$  &805\\
\cline{2-5}$4$& TRA &105.40 &$6.177\times 10^{-9}$  &1130\\ \cline{2-5}
& TRHA &95.48 &$4.213\times 10^{-9}$  &1030\\\hline
\hline & phipm  &597.95 &--  &10008\\
\cline{2-5}& $\rm funm_-kryl$  &670.04 &--  &3995\\
\cline{2-5}total& TRA &505.47 &--  &5395\\ \cline{2-5}
& TRHA &461.55 &--  &4945\\\hline
\end{tabular}
\end{center}
{Example 6.2.~Table 3: Numerical results of the problem (\ref{6.2}), the matrix size $n=N^2=250,000$. Compute $\varphi_\ell(-t\hat{A}){\bf v},~\ell=1,2,3,4$ {\it sequentially} (one by one)
by using {\it phipm}, $\it funm_-kryl$, TRA and TRHA.

%\vspace{0.2cm}
\textbf{Example 6.2.}~ In this example, we consider the following stiff problems \cite{Goc}
\begin{equation}\label{6.2}
{\bf y}'(t)=-A{\bf y}(t)+\frac{t^{\ell-1}}{(\ell-1)!}{\bf v},\quad {\bf y}(0)={\bf 0},\quad \ell=1,2,\ldots,s.
\end{equation}
The exact solutions at time $t$ are ${\bf y}(t)=t^{\ell}\varphi_\ell(-t A){\bf v},~\ell=1,2,\ldots,s$.
Therefore, we need to solve the $s$ vectors simultaneously. In this example,
the matrix $-A\in \mathbb{R}^{n\times n}$ is the standard finite difference discretization matrix for the
two-dimensional Laplacian on the unit square with homogeneous Dirichlet boundary conditions, where we use
a regular grid with $n=N^{2}$ inner discretization points and mesh size $h=\frac{1}{N+1}$. The vector ${\bf v}=\big(g(ih,jh)\big)_{i,j=1}^{N}$ contains
the evaluations of the function $g(x,y)=30x(1-x)y(1-y)$ at the inner grid points.
We compute ${\bf y}(t)=\varphi_\ell(-t\hat{A}){\bf v},~\ell=1,2,3,4$ with $\hat{A}=0.025\times A$ and $t=1$. The ``exact" solutions
are obtained from running {\it phipm} with the convergence tolerance $tol=10^{-14}$.

It is seen from Tables 3--4 that both TRA and TRHA outperform the other algorithms, and we benefit from the thick-restarting strategy. Specifically, when the vectors are computed simultaneously, TRHA works much better than $\it funm_-kryl$ and {\it phipm} in terms of CPU time, 154.09 seconds vs. 462.72 and 597.95 seconds, a great improvement. On the other hand, we observe from the two tables that TRHA performs better than TRA in terms of both CPU time and the number of matrix-vector products. Furthermore, the accuracy of the approximations got from TRHA is a little higher than that of TRA. All these show the superiority of the harmonic projection technique over the orthogonal projection technique for $\varphi$-functions.

\begin{center}
\begin{tabular}{|c|c|c|c|c|}
\multicolumn{4}{c}{{}}\\[3pt]\hline
~$\ell$~&~{\bf Algorithm}~&~{\bf CPU}~&~{\bf Error}~&~{\bf Mv}\\
\hline  &   & &$1.338\times 10^{-11}$  &\\
$1\sim4$& $\rm funm_-kryl$ &462.72 &$1.204\times 10^{-12}$  &1205\\
&  & &$2.249\times 10^{-13}$  &\\
&  & &$1.395\times 10^{-13}$  &\\\hline
\hline  &   & &$1.946\times 10^{-10}$  &\\
 $1\sim4$  & TRA &164.71 &$3.126\times 10^{-10}$  &1555\\
&  & &$1.520\times 10^{-9}$  &\\
&  & &$6.177\times 10^{-9}$  &\\\hline
\hline  &   & &$1.122\times 10^{-10}$  &\\
$1\sim4$& TRHA &154.09 &$1.908\times 10^{-10}$  &1455\\
&  & &$9.217\times 10^{-10}$  &\\
&  & &$4.213\times 10^{-9}$  &\\\hline
\end{tabular}\\[5pt]
\end{center}
{Example 6.2.~Table 4: Numerical results of the problem (\ref{6.2}), the matrix size $n=N^2=250,000$. Compute $\varphi_\ell(-t\hat{A}){\bf v},~\ell=1,2,3,4$ {\it simultaneously} by using $\it funm_-kryl$, TRA and TRHA.}

%\vspace{0.2cm}
\textbf{Example 6.3.}~~This test problem is the {\it $G2_-circuit$} matrix arising
from the circuit simulation problem. It is size of $150,102 \times 150,102$, with $726,674$ nonzero elements, whose data file is available from the University of Florida Sparse Matrix
Collection: {\it http://www.cise.ufl.edu/research/sparse/matrices}.
In this example, we want to compute $\varphi_\ell(-A){\bf v},~\ell=0,1,2,3$ with $A=10 \times G2_-circuit$ and ${\bf v}=[1,1,\ldots,1]^{\rm T}$, by using {\it phipm}, $\rm funm_-kryl$, TRA and TRHA.
The ``exact" solutions are got from running {\it phipm} with the convergence tolerance $tol=10^{-14}$.
Tables 5 and 6 list the numerical results.

Again, it is observed from Tables 5--6 that TRHA works much better than the other algorithms
in terms of CPU time, especially when the vectors are computed simultaneously. However, the accuracy of the approximations obtained from {\it phipm} and $\it funm_-kryl$
can be (much) higher than that obtained from TRA and TRHA. As we have mentioned before, the reason is due to the fact that one needs to solve a small-sized ODE problem inexactly during each cycle of the two residual-based algorithms. Therefore, if accuracy is not the most important thing and one wants to solve (\ref{eqn1.3}) rapidly, TRHA is a competitive candidate for the $\varphi$-functions of very large matrices.

\begin{center}
\begin{tabular}{|c|c|c|c|c|}
\multicolumn{4}{c}{{}}\\[3pt]\hline
~$\ell$~&~{\bf Algorithm}~&~{\bf CPU}~&~{\bf Error}~&~{\bf Mv}\\
\hline & phipm  &121.43 &$8.843\times 10^{-14}$  &4982\\
\cline{2-5} & $\rm funm_-kryl$  &445.84 &$1.573\times 10^{-11}$  &1880\\
\cline{2-5}$0$ & TRA &153.04 &$6.406\times 10^{-9}$  &2780\\ \cline{2-5}
& TRHA &125.48 &$8.892\times 10^{-9}$  &2380\\  \hline
\hline & phipm  &130.30 &$2.890\times 10^{-14}$  &3144\\
\cline{2-5} & $\rm funm_-kryl$  &689.39 &$7.898\times 10^{-12}$  &1755\\
\cline{2-5}$1$& TRA &137.60 &$2.401\times 10^{-9}$  &2630\\ \cline{2-5}
& TRHA &110.82 &$1.668\times 10^{-9}$  &2230\\  \hline
\hline & phipm  &115.80 &$6.714\times 10^{-13}$  &2820\\
\cline{2-5} & $\rm funm_-kryl$  &763.95 &$1.967\times 10^{-11}$  &1605\\
\cline{2-5}$2$& TRA &126.84 &$4.788\times 10^{-9}$  &2480\\ \cline{2-5}
& TRHA &100.34 &$2.982\times 10^{-9}$  &2080\\\hline
\hline & phipm  &107.20 &$6.643\times 10^{-12}$  &2621\\
\cline{2-5} & $\rm funm_-kryl$  &698.22 &$1.020\times 10^{-10}$  &1430\\
\cline{2-5}$3$& TRA &115.23 &$1.259\times 10^{-8}$  &2280\\ \cline{2-5}
& TRHA &90.58 &$8.306\times 10^{-9}$  &1905\\\hline
\hline & phipm  &474.74 &--  &13567\\
\cline{2-5}& $\rm funm_-kryl$  &2597.40 &--  &6670\\
\cline{2-5}total& TRA &532.71 &--  &10170\\ \cline{2-5}
& TRHA &427.21 &--  &8595\\\hline
\end{tabular}
\end{center}
{Example 6.3.~Table 5: Numerical results of computing $\varphi_\ell(-A){\bf v},~\ell=0,1,2,3$ {\it sequentially} (one by one)
by using {\it phipm}, $\it funm_-kryl$, TRA and TRHA. The matrix $A=10\times G2_-circuit$, which is of size $n=150,102$.}

\begin{center}
\begin{tabular}{|c|c|c|c|c|}
\multicolumn{4}{c}{{}}\\[3pt]\hline
~$p$~&~{\bf Algorithm}~&~{\bf CPU}~&~{\bf Error}~&~{\bf Mv}\\
\hline  &   & &$1.573\times 10^{-11}$  &\\
$0\sim3$& $\rm funm_-kryl$ &2220.20 &$6.834\times 10^{-13}$  &1880\\
&  & &$4.309\times 10^{-13}$  &\\
&  & &$3.169\times 10^{-13}$  &\\\hline
\hline  &   & &$6.406\times 10^{-9}$  &\\
 $0\sim3$  & TRA &187.27 &$2.401\times 10^{-9}$  &2780\\
&  & &$4.788\times 10^{-9}$  &\\
&  & &$1.259\times 10^{-8}$  &\\\hline
\hline  &   & &$8.892\times 10^{-9}$  &\\
$0\sim3$& TRHA &157.04 &$1.667\times 10^{-9}$  &2380\\
&  & &$2.982\times 10^{-9}$  &\\
&  & &$8.306\times 10^{-9}$  &\\\hline
\end{tabular}\\[5pt]
\end{center}
{Example 6.3.~Table 6: Numerical results of computing $\varphi_\ell(-A){\bf v},~\ell=0,1,2,3$ {\it simultaneously}
by using $\it funm_-kryl$, TRA and TRHA. The matrix $A=10\times G2_-circuit$, which is of size $n=150,102$.}

\vspace{0.2cm}\textbf{Example 6.4.}~ In this example, the test matrix is generated by using the MATLAB function ``\texttt{gallery}":
$A=-gallery('lesp',6000)$. It returns a $6000\times 6000$ tridiagonal matrix with real, sensitive eigenvalues.
We compute $\varphi_{\ell}(-A)\mathbf{v},~\ell=1,2,3,4$ by {\it phipm}, $\it funm_-kryl$, TRA and TRHA,
where ${\bf v}$ is set to be the vector of all ones.
The ``exact" solutions are derived from running the MATLAB function {\it phipade} of the EXPINT package \cite{Ber2}. Tables 7--8 list the numerical results.

It is seen from the numerical results that TRHA still works quite well for the matrix with sensitive eigenvalues. Indeed, TRA and TRHA outperform {\it phipm} and $\it funm_-kryl$ considerably in terms of CPU time, while TRHA
performs the best in many cases. Furthermore, one can save about one half of CPU time if the 4 vectors are computed simultaneously instead of sequentially. For this test problem, if the vectors are evaluated one by one, we observe from Table 7 that the accuracy of the approximations obtained from TRA and TRHA is comparable to that of the approximations from running $\it funm_-kryl$.

\begin{center}
\begin{tabular}{|c|c|c|c|c|}
\multicolumn{4}{c}{{}}\\[3pt]\hline
~$\ell$~&~{\bf Algorithm}~&~{\bf CPU}~&~{\bf Error}~&~{\bf Mv}\\
\hline & phipm  &50.03 &$2.824\times 10^{-10}$  &2572\\
\cline{2-5} & $\rm funm_-kryl$  &110.79 &$6.538\times 10^{-9}$  &1055\\
\cline{2-5}$1$ & TRA &31.88 &$2.546\times 10^{-8}$  &1280\\ \cline{2-5}
& TRHA &30.54 &$1.508\times 10^{-8}$  &1205\\  \hline
\hline & phipm  &44.67 &$1.986\times 10^{-10}$  &2273\\
\cline{2-5} & $\rm funm_-kryl$  &108.58 &$1.096\times 10^{-8}$  &930\\
\cline{2-5}$2$& TRA &27.16 &$3.085\times 10^{-8}$  &1180\\ \cline{2-5}
& TRHA &25.97 &$2.343\times 10^{-8}$  &1105\\  \hline
\hline & phipm  &42.36 &$1.209\times 10^{-9}$  &2131\\
\cline{2-5} & $\rm funm_-kryl$  &96.94 &$1.864\times 10^{-8}$  &830\\
\cline{2-5}$3$& TRA &24.45 &$1.210\times 10^{-7}$  &1080\\ \cline{2-5}
& TRHA &22.93 &$7.545\times 10^{-8}$  &1005\\\hline
\hline & phipm  &42.78 &$8.563\times 10^{-10}$  &2145\\
\cline{2-5} & $\rm funm_-kryl$  &77.03 &$4.985\times 10^{-8}$  &730\\
\cline{2-5}$4$& TRA &21.57 &$5.240\times 10^{-7}$  &955\\ \cline{2-5}
& TRHA &19.89 &$3.367\times 10^{-7}$  &880\\\hline
\hline & phipm  &179.84 &--  &9121\\
\cline{2-5}& $\rm funm_-kryl$  &393.34 &--  &3545\\
\cline{2-5}total& TRA &105.05 &--  &4495\\ \cline{2-5}
& TRHA &99.33 &--  &4195\\\hline
\end{tabular}
\end{center}
{Example 6.4.~Table 7: Numerical results of computing $\varphi_\ell(-A){\bf v},~\ell=1,2,3,4$ {\it sequentially} (one by one)
by using {\it phipm}, $\it funm_-kryl$, TRA and TRHA. The matrix $A=-gallery('lesp',6000)$, which is of size $n=6000$.}\\

\begin{center}
\begin{tabular}{|c|c|c|c|c|}
\multicolumn{4}{c}{{}}\\[3pt]\hline
~$\ell$~&~{\bf Algorithm}~&~{\bf CPU}~&~{\bf Error}~&~{\bf Mv}\\
\hline  &   & &$6.538\times 10^{-9}$  &\\
$1\sim4$& $\rm funm_-kryl$ &267.98 &$4.259\times 10^{-10}$  &1055\\
&  & &$4.617\times 10^{-11}$  &\\
&  & &$6.551\times 10^{-12}$  &\\\hline
\hline  &   & &$2.546\times 10^{-8}$  &\\
 $1\sim4$  & TRA &42.29 &$3.088\times 10^{-8}$  &1280\\
&  & &$1.209\times 10^{-7}$  &\\
&  & &$5.240\times 10^{-7}$  &\\\hline
\hline  &   & &$1.508\times 10^{-8}$  &\\
$1\sim4$& TRHA &40.49 &$2.364\times 10^{-8}$  &1205\\
&  & &$7.524\times 10^{-8}$  &\\
&  & &$3.364\times 10^{-7}$  &\\\hline
\end{tabular}\\[5pt]
\end{center}
{Example 6.4.~Table 8: Numerical results of computing $\varphi_\ell(-A){\bf v},~\ell=1,2,3,4$ {\it simultaneously}
by using $\it funm_-kryl$, TRA and TRHA. The matrix $A=-gallery('lesp',6000)$, which is of size $n=6000$.}\\

Since the computation of $\varphi$-functions can be rewritten in terms of a single matrix exponential by considering a slightly augmented matrix \cite{MA,Saad,Sidje}, it is interesting to investigate the numerical approximation of the matrix exponential applied to a vector. In the following two examples, we try to compare our TRHA algorithm with some state-of-the-art algorithms, such as {\it expv}, {\it phipm}, ${\it funm_-kryl}$ and ${\it funm_-quad}$ for matrix exponential.

\vspace{0.2cm}\textbf{Example 6.5.}~
In this example, we compare our TRHA algorithm with {\it expv, ${\it funm_-kryl, funm_-quad}$, ${\it Rich_-Kry}$}
and TRA, and show the efficiency of the new algorithm
for matrix exponential. In this example, the ``exact" solutions are derived from running the MATLAB built-in function {\it expm.m}.

There are two test problems in this example.
The first one is from \cite{Pang}. We consider pricing options for a single underlying asset in Merton's
jump-diffusion model \cite{RM-1976}. In
Merton's model, jumps are normally distributed with mean $\hat{\mu}$ and variation $\sigma$. The
option value $w(\xi,\tau)$ with logarithmic price $\xi$ and backward time $\tau$ satisfies a forward PIDE on $(-\infty,+\infty)\times [0,t]$:
\begin{equation} \label{equ1.2}
w_{\tau}=\frac{\nu^{2}}{2}w_{\xi\xi}+(r-\hat{\lambda}\kappa-\frac{\nu^{2}}{2})w_{\xi}-(r+\hat{\lambda})w+\hat{\lambda}\int^{\infty}_{-\infty}w(\xi+\eta,\tau)\phi(\eta)d\eta,
\end{equation}
where $t$ is the maturity time, $\nu$ is the stock return volatility, $r$ is the risk-free interest
rate, $\hat{\lambda}$ is the arrival intensity of a Poisson process, $\kappa=\exp(\hat{\mu}+\frac{\sigma^{2}}{2})-1$ is the expectation of the impulse function,
and $\phi$ is the Gaussian distribution given by
\begin{equation*}
\phi(\eta)=\frac{\exp(-(\eta-\hat{\mu})^{2})/2\sigma^{2}}{\sqrt{2\pi}\sigma}.
\end{equation*}
For a European call option, the initial condition is
\begin{equation}\label{equ1.3}
w(\xi,0)=\max(K\exp(\xi)-K,0),
\end{equation}
where $K$ is the strike price \cite{RM-1976}.  Similar to \cite{Pang}, we truncate the $\xi$-domain $(-\infty,\infty)$ to $[\xi_{1},\xi_{2}]$
and then divide $[\xi_{1},\xi_{2}]$ into $n+1$ subintervals with a uniform mesh size $h_{\xi}$.
By approximating the differential part of (\ref{equ1.2}) by central difference discretization, we can obtain an $n \times n$ tridiagonal Toeplitz matrix
$\mathcal{D}_{n}$. For the integral term in (\ref{equ1.2}), the localized part can be expressed in discrete form by using the rectangle rule. The corresponding
operator is an $n \times n$ Toeplitz matrix $\mathcal{I}_{n}$. Then
the real nonsymmetric Toeplitz matrix $A=-\mathcal{D}_{n}-\hat{\lambda}\mathcal{I}_{n}$
is the coefficient matrix of the semidiscretized system with regard to $\tau$. The option
price at $\tau=t$ requires evaluating the exponential term $\exp(-tA){\bf w}$, where ${\bf w}$ is the discretized form of the initial value in (\ref{equ1.3});
see \cite{Pang} for more details. The input parameters of this problem are $\xi_{1}=-2$, $\xi_{2}=2$, $K=100$, $\nu=0.25$, $r=0.05$, $\hat{\lambda}=0.1$, $\hat{\mu}=-0.9$ and $\sigma=0.45$.
Table 9 presents the numerical results of this problem when $t=1$ and $n=4000$.\\

\begin{center}%\footnotesize
\begin{tabular}
{|c|c|c|c|c|}
\hline
  {\bf Algorithm} &  \lower.4ex\hbox{{\bf CPU}} &  \lower.4ex\hbox{{\bf Error}} & \lower.4ex\hbox{{\bf Mv}} \\ \hline
 expv &188.09&$1.173\times 10^{-11}$ &23188\\
 \hline
 ${\rm funm_-kryl}$ &3894.90&$1.191\times 10^{-11}$  &3505 \\
 \hline
 ${\rm funm_-quad}$ &n.c.&n.c.   &n.c.\\
 \hline
 ${\rm Rich_-Kryl}$ &202.56&$1.990\times 10^{-8}$  &5700\\
 \hline
  TRA &118.27&$9.003\times 10^{-9}$  &5430 \\
 \hline
  TRHA &115.83&$1.836\times 10^{-9}$ &4805\\
 \hline
\end{tabular}\\[5pt]
\end{center}
{Example 6.5.~Table 9: Numerical results of the six algorithms on the first test problem, where ``n.c." denotes ``fails to converge".}\\

The second test problem is from numerical solution of the
following fractional diffusion equation \cite{Pang-Sun-2012,WWS-2010}:
\begin{equation}\label{equ1.1}
\left\{\begin{array}{l}
\frac{\partial u(x,t)}{\partial t}-d_{+}(x)\frac{\partial ^{\alpha}u(x,t)}{\partial_{+}x^{\alpha}}-
 d_{-}(x)\frac{\partial ^{\alpha}u(x,t)}{\partial_{-}x^{\alpha}}=f(x,t), \\
x\in (0,2), ~~t\in (0,1],\\
u(0,t)=u(2,t)=0, ~~t\in [0,1],\\
u(x,0)=u_{0}(x), ~~x\in [0,2].
 \end{array}\right.
\end{equation}
In this equation, we set the the coefficients
$$d_{+}(x)=\frac{Gamma(3-\alpha)}{100}x^{\alpha},$$ and
$$d_{-}(x)=\frac{Gamma(3-\alpha)}{100}(2-x)^{\alpha},$$
where $1 < \alpha < 2$ and $Gamma$ is the Gamma function. We refer to \cite{I-1999} for the definition of the fractional
order derivative. After spatial discretization by using the shifted Gr\"{u}nwald formula \cite{M-C-1999},
the equation
(\ref{equ1.1}) reduces to a semidiscretized ordinary differential equation with the coefficient matrix $A_{h}=-\frac{1}{h^{\alpha}}(D_{+}G+D_{-}G^{{\rm T}})$, where
 $h$ is the spatial grid size,
 $D_{+}$ and $D_{-}$ are diagonal matrices arising from the discretization of the diffusion coefficient $d_{+}(x)$ and $d_{-}(x)$, and $G$
is a lower Hessenberg Toeplitz matrix generated by the discretization of the fractional derivative.
In this experiment, we choose $\alpha=1.8$, $t=1$, and compute ${\rm exp}(-tA_{h}){\bf v}$ with $\bf v$ being the vector of all ones and the size of the matrix being $n=4000$. Table 10 lists the numerical results.

\begin{center}
\begin{tabular}
{|c|c|c|c|c|}
\hline
  {\bf Algorithm} &  \lower.4ex\hbox{{\bf CPU}} &  \lower.4ex\hbox{{\bf Error}} & \lower.4ex\hbox{{\bf Mv}} \\ \hline
  expv &130.21&$4.717\times 10^{-11}$ &15872\\
  \hline
  $\rm funm_-kryl$ &2154.30&$1.230\times 10^{-10}$  &3030 \\
 \hline
 $\rm funm_-quad$ &63.72&$4.320\times 10^{-11}$   &3230\\
 \hline
 $\rm Rich_-Kryl$ &109.28&$1.066\times 10^{-7}$  &5280\\
 \hline
  TRA &73.37&$3.009\times 10^{-8}$  &4555 \\
 \hline
  TRHA &69.96&$1.941\times 10^{-8}$ &4080\\
 \hline
\end{tabular}\\[5pt]
{Example 6.5.~Table 10: Numerical results of the six algorithms on the second test problem.}
\end{center}

Two remarks are given. First, we see from Tables 9 and 10 that TRHA outperforms the other algorithms in terms of CPU time in most cases.
In particular, we observe that TRA and TRHA perform much better than $\it Rich_-Kry$ in terms of CPU time and the number of matrix-vector products. This shows that the convergence speed of the Krylov subspace algorithm can be improved significantly by using the thick-restarting strategy.
Second, similar to the above numerical experiments, we notice that the accuracy of the approximations obtained from {\it expv}, $\it funm_-quad$ and $\it funm_-kryl$ can be (much) higher than that from the residual based algorithms $\it Rich_-Kry$, TRA and TRHA. As we have mentioned before, the reason is due to the fact that we have to solve a small-sized ODE problem with the tolerance being $10^{-9}$ in each cycle of the residual-based algorithms.

\vspace{0.2cm}\textbf{Example 6.6.}~ In this example, we consider the matrix exponential problem of a large matrix.
The test matrix $A$ is the {\it apache1} matrix arising from the structural problem. It is  of size $80800 \times 80800$, with $542184$ nonzero elements. The data file is available from the University of Florida Sparse Matrix
Collection: {\it http://www.cise.ufl.edu/research/sparse/matrices}.
In this example, we try to evaluate ${\rm exp}(-A){\bf v}$ with ${\bf v}$ being the vector of all ones. Table 11 reports the numerical results.

\begin{center}%\footnotesize
\begin{tabular}
{|c|c|c|c|c|}
\hline
 {\bf Algorithm}  &  \lower.4ex\hbox{{\bf CPU}} &  \lower.4ex\hbox{{\bf Error}} & \lower.4ex\hbox{{\bf Mv}} \\ \hline
  expv &74.44&$8.424\times 10^{-12}$ &11036\\
 \hline
 $\rm funm_-kryl$ &695.72&$6.719\times 10^{-11}$  &2155 \\
 \hline
  $\rm funm_-quad$ &n.c.&n.c.   &n.c.\\
 \hline
 $\rm Rich_-Kry$ &83.50&$5.688\times 10^{-9}$  &3600\\
 \hline
  TRA &85.34&$1.627\times 10^{-9}$  &3430 \\
 \hline
  TRHA &73.38&$1.003\times 10^{-9}$ &3005\\
 \hline
\end{tabular}\\[5pt]
\end{center}
{Example 6.6.~Table 11: Numerical results of the six algorithms on ${\rm exp}(-A){\bf v}$, where ``n.c." denotes ``fails to converge".}\\

As the MATLAB function {\it expm.m} is infeasible
for very large matrices, in this example, we use the MATLAB function ${\it expmv_-tspan}$ \cite{A.H.} to compute the ``exact" solution of the exponential function,
The MATLAB codes are available from {\it http://www.maths.manchester.ac.uk/\textasciitilde almohy/ papers.html}.
Again, the numerical results show that our new algorithm is superior to
the other algorithms in terms of CPU time, and the residual-based TRA and TRHA algorithms are suitable for exponential of very large matrices.
Specifically, TRHA performs much better than the two deflated Krylov subspace algorithms $\it funm_-kryl$ and $\it funm_-quad$.
Similar to the above numerical examples, we remark that the number of matrix-vector products is not the whole story for
the matrix exponential problem. For instance, we notice that TRHA used $3005$ matrix-vector products and $73.38$ seconds, while {\it expv} used
$11036$ matrix-vector products and $74.44$ seconds. The reason is that one requires to solve an ODE problem in each outer iteration (cycle) of the TRHA algorithm. Experimentally, we find that if the number of restarting is large, solving the ODE problems during cycles will bring us a large amount of computational overhead.

%\section{Conclusions}
%
%In order to solve the small sized ODE (\ref{eqn3.41}), we have to store some values of
%right-hand side for different $t$ spanning the time interval of interest. In practical experiments, we find that (\ref{eqn3.41}) is difficult to be solved numerically when $t$ approaches to zero in the time interval of interest.

%\section*{Acknowledgments}

\end{document}